\documentclass[10pt]{amsart}
\usepackage{amsmath,amssymb,latexsym,cancel,rotating}
\usepackage{tikz-cd}
\usepackage{tikz}
\usepackage{graphicx,mathrsfs,color,fancyhdr,amsthm, float}
\usepackage[all]{xy}
\usepackage{geometry}
\usepackage{hyperref}
\usepackage{setspace}
\newtheorem{theorem}{Theorem}[section]
\newtheorem{lemma}[theorem]{Lemma}
\newtheorem{corollary}[theorem]{Corollary}
\newtheorem{definition}[theorem]{Definition}
\newtheorem{proposition}[theorem]{Proposition}

\newtheorem{remark}[theorem]{Remark}

 \def\cD{{\mathcal D}}    
\def\cH{{\mathcal H}}     
 \def\cO{{\mathcal O}}

\def\bbN{{\mathbb N}}  \def\bbZ{{\mathbb Z}}  \def\bbQ{{\mathbb Q}}
    \def\bbF{{\mathbb F}}
\def\bbC{{\mathbb C}}    \def\bbP{{\mathbb P}}

         \def\bfU{{\bf U}}
   
  \def\leq{\leqslant}  \def\geq{\geqslant}

\def\Hom{\mbox{\rm Hom}}  
       \def\im{\mbox{\rm Im}\,}
\def\Ext{\mbox{\rm Ext}}   
\def\dim{\mbox{\rm dim}}   
\def\Aut{\mbox{\rm Aut}}

  \def\tr{\mbox{\rm tr}}

\def\bfV{{\mathbf V}}

\def\bfH{{\mathbf H}}

\def\supp{{\rm supp}}

\def\bfV{\mathbf{V}}
\def\bfW{\mathbf{W}}
\def\bfE{\mathbf{E}}
\def\bfG{\mathbf{G}}

\def\bfU{\mathbf{U}}

\def\Ind{\mathbf{Ind}}

\newcommand\bbq{{\mathbb Q}}

\newcommand\mk{{\mathcal{K}}}

\newcommand\mf{{\mathcal{F}}}

\newcommand\mo{{\mathcal{O}}}

\newcommand\ml{{\mathcal{L}}}

\begin{document}

\title[]{Geometric realization of affine bases: the Kronecker quiver case}
\author{Yumeng Wu, Jie Xiao}
\address{Beijing International Center for Mathematical Research, Beijing 100871, P. R. China}
\email{2506397175@pku.edu.cn (Y. Wu)}
\address{School of Mathematical Sciences, Beijing Normal University, Beijing 100875, P. R. China}
\email{jxiao@bnu.edu.cn (J. Xiao)}



\bibliographystyle{abbrv}
\begin{abstract}

In this paper, we study the transition matrix between the PBW basis and the canonical basis for the negative part of the quantized enveloping algebra of the Kronecker quiver from a geometric viewpoint. Building on Lusztig’s geometric construction of the canonical basis, we construct sheaf-complex realizations of PBW basis elements by means of flag sheaf complexes over the strata $X(\alpha,m)$ of representation varieties. 

Our first goal is to give a geometric description of the simple constituents appearing in the restrictions of these flag sheaf complexes to the strata $X(\alpha,m)$. This allows us to compare the PBW-type sheaf complexes with the simple perverse sheaves $IC(X(\alpha),L_\chi)$ arising in Lusztig’s construction. Using this description together with a purity result for the relevant $\bbF_q$-structures, we obtain another proof that the elements defined by Lusztig’s perverse sheaves indeed form a basis of the composition algebra.

Our second goal is to make the transition coefficients between the PBW basis and the canonical basis geometrically explicit. More precisely, we show that these coefficients are governed by the multiplicities of local systems in the restrictions of intersection cohomology complexes to smaller strata. As a consequence, the transition matrix from the canonical basis to the PBW basis is upper triangular with diagonal entries equal to $1$, and its coefficients admit a direct geometric interpretation. In particular, in the Kronecker quiver case we recover the triangularity of the transition matrix and obtain positivity properties of the corresponding coefficient polynomials.

\end{abstract}

\maketitle
\setcounter{tocdepth}{1}\tableofcontents
\begin{spacing}{1.5}

\section*{Introduction}

Lusztig constructed the negative part of the quantized enveloping algebra associated with a Dynkin quiver by geometric methods, and realized its canonical basis in terms of simple perverse sheaves in \cite{lusztig1990canonical}. In the same work, he also gave a geometric interpretation of the transition matrix between a PBW basis and the canonical basis. Later, in \cite{Lusztig1992}, Lusztig developed a geometric construction of the negative part of the quantized enveloping algebra and canonical basis associated with a tame quiver (see more details in \cite{li2007notesaffinecanonicalmonomial}). It is therefore natural to seek an extension of the methods and results of \cite{lusztig1990canonical} from Dynkin quivers to tame quivers.

There has been substantial work on the construction of PBW bases and canonical bases via representations of tame quivers. A comprehensive construction of PBW bases of affine type was provided by Beck and Nakajima by purely algebraic methods \cite{2002Crystal}. In the case of the Kronecker quiver, McGerty identified the restriction of sheaf complexes corresponding to the imaginary part of a PBW basis \cite{MCGERTY2005411}, which is the main motivation for our work. Related studies of PBW bases and their relation to canonical bases by KLR algebra methods can be found in \cite{10.1215/00127094-2681278} and \cite{McNamara2017}. Related work on canonical bases and PBW bases from representations of tame quivers has been carried out by Lin--Xiao--Zhang\cite{linxiao}, Li--Lin\cite{Li2006ARquiverAT}, Xiao–Xu–Zhao\cite{XIAO2023510}, and Shoji--Zhou\cite{shoji2026monomialbasescanonicalbases}.

In the Dynkin case, using Lusztig’s geometric construction, Kato constructed sheaf complexes corresponding to PBW bases and provided a model for understanding the relationship between the PBW basis and the canonical basis \cite{cd088214-009e-3813-9a86-97f9a299a24c,2012PBW}. The Kronecker quiver is the first and most fundamental example beyond Dynkin type, and serves as a testing ground for the general tame case. The aim of this paper is to develop a geometric construction of PBW-type bases for the Kronecker quiver and to compare them with the canonical basis realized by simple perverse sheaves. In particular, we seek to give a geometric interpretation of the transition coefficients between these bases. More precisely, for the strata $X(\alpha,m)$ (see Definition~\ref{deff}) in the representation varieties, we consider flag sheaf complexes obtained from Lusztig’s induction functors and analyze their restrictions to these strata. Our first main result of the paper shows that every simple constituent appearing in the perverse cohomology of the restricted flag sheaf complexes is of the form $j_{\alpha',m'}^*IC(X(\alpha'),L_\chi)$. This gives a precise geometric description of the relevant simple constituents and enables us to control the comparison between the two bases. 

Our second result is a purity statement for the $\bbF_q$-structures underlying these complexes. In this paper, this is formulated as the pointwise purity of the flag sheaf complex $\mathcal{L}_0$, together with the purity of its restrictions. This purity result allows us to pass from geometric decompositions to trace functions and hence obtain the precise relations between the PBW basis and the canonical basis in the quantized enveloping algebra. In particular, it provides another route to prove that the elements corresponding to the perverse sheaves $IC(X(\alpha),L_\chi)$ form a basis. 

After establishing these geometric and purity properties, we turn to the transition coefficients between the canonical basis and the PBW basis. The main point is that these coefficients are not merely formal structure constants; rather, they admit a direct geometric interpretation. More precisely, they are expressed in terms of the multiplicities of local systems $L_{\chi'}$ occurring in the restrictions of $IC(X(\alpha),L_\chi)$ to smaller strata. This gives a geometric realization of the coefficient polynomials $p_{(\alpha,\chi),(\alpha',\chi')}(v)$. As a consequence, we obtain that the transition matrix is upper triangular with respect to the natural order on $(\alpha,m)$, that its diagonal entries are equal to $1$, and that its off-diagonal coefficients satisfy positivity. 

The paper is organized as follows. In Section 1, we review the representation theory of the Kronecker quiver and the stratification of the corresponding moduli spaces by the subsets $X(\alpha,m)$. In Section 2, we recall Lusztig’s construction and the relevant flag sheaf complexes, together with the explicit description of some flag sheaves on $X(\alpha,m)$. In Section 3, we study the restrictions of arbitrary flag sheaf complexes to $X(\alpha,m)$ and prove the key description of their simple constituents. In Section 4, we establish the purity of the corresponding complexes. In Section 5, we use these results to derive the geometric formula for the coefficients between the PBW basis and the canonical basis, as well as the triangularity and positivity properties of the associated polynomials. The final two sections are devoted to the proofs of the main technical results.

\section{Stratification of moduli spaces for the Kronecker quiver}

The Kronecker quiver is a particularly important example among tame quivers. A detailed study of the representations of the Kronecker quiver provides a new perspective on the relationship between the canonical basis and the PBW basis. In this section, we recall several important properties of the representations of the Kronecker quiver, and end with a stratification of the corresponding moduli space. Throughout, we assume that the base field is the algebraic closure $k=\overline{\bbF_q}$ of the finite field $\bbF_q$, where $\operatorname{char} k=p$ and $q=p^r$ for some $r\in \bbN_{>0}$.

\begin{definition}
A Kronecker quiver is a quiver $Q=(I,H, s,t),$ where
\[
I=\{1,2\},\qquad
H=\{h_1,h_2\},\qquad
\]
and
\[
s(h_i)=1,\quad t(h_i)=2
\qquad (i=1,2).
\]
\end{definition}

This quiver can be depicted as
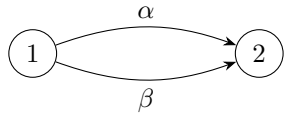
\[
\begin{tikzpicture}[>=Stealth, node distance=3cm]
 \tikzstyle{vertex}=[circle,draw,minimum size=18pt,inner sep=0pt]
 \node[vertex] (1) {$1$};
 \node[vertex] (2) [right of=1] {$2$};
 \draw[->] (1) to[bend left=20] node[above] {$\alpha$} (2);
 \draw[->] (1) to[bend right=20] node[below] {$\beta$} (2);
\end{tikzpicture}
\]
The representations of the Kronecker quiver $Q$ are of the form $(\bfV, x)$, where $\bfV=\bfV_1\oplus\bfV_2$ and
\[
x\in \bigoplus_{h\in H}\Hom_k(\bfV_{s(h)},\bfV_{t(h)})=: \bfE_{\bfV}.
\]
We denote the representation corresponding to $(\bfV, x)$ as $M_x.$

Here, $\bfE_{\bfV}$ is the corresponding representation space for $\bfV$, and we set $$\bfG_{\bfV}=\Pi_{i\in I}\operatorname{GL}_k(\bfV_i).$$

$\bfG_{\bfV}$ acts on $\bfE_{\bfV}$ as $g((x_h)_{h\in H})=(g_{t(h)}(x_h)g_{s(h)}^{-1})_{h\in H},$ for $g\in \bfG_{\bfV}$ and $(x_h)_{h\in H}\in \bfE_{\bfV}.$
Both $\bfE_{\bfV}$ and $\bfG_{\bfV}$ carry canonical $\bbF_q$-structures.

\begin{definition}
The isomorphism classes of indecomposable preprojective representations of $Q$ of dimension $(s-1,s)$ can be indexed by $s\in \bbN_{\geq 1}$. We denote by $P_s$ a representative of the indecomposable preprojective representation corresponding to $s$.

The isomorphism classes of indecomposable preinjective representations of $Q$ of dimension $(s,s-1)$ can be indexed by $s\in \bbN_{\geq 1}$. We denote by $I_s$ a representative of the indecomposable preinjective representation corresponding to $s$.
\end{definition}

Set
\[
R:=\bigoplus_{s\in \bbN_{\geq 1}}\bbN P_s \;\oplus\; \bigoplus_{s\in \bbN_{\geq 1}}\bbN I_s.
\]
Let $\alpha\in R$ be of the form $\alpha=\sum_s \alpha_s^P P_s+\sum_s \alpha_s^I I_s$, and assume that
\[
\dim \bfV-\sum_s \alpha_s^P(s-1,s)-\sum_s \alpha_s^I(s,s-1)=(m,m).
\]
We also write
\[
P_{\alpha}=\bigoplus_s P_s^{\alpha_s^P},
\qquad
I_{\alpha}=\bigoplus_s I_s^{\alpha_s^I}.
\]

\begin{lemma}\cite{Lusztig1992}\label{le}
For $s,s'\in \bbN$ with $s'>s$, we have $\Hom(P_{s'},P_s)=0$ and $\Ext^1(P_s,P_{s'})=0$.

For $s,s'\in \bbN$ with $s'>s$, we have $\Hom(I_s,I_{s'})=0$ and $\Ext^1(I_{s'},I_s)=0$.

For any $s,s'\in \bbN$ and any indecomposable regular representation $H_r$, we have
\[
\Hom(I_{s'},\, H_r\oplus P_s)=0,\qquad \Hom(H_r,P_s)=0,
\]
and
\[
\Ext(P_s,\, I_{s'}\oplus H_r)=0,\qquad \Ext(H_r,I_{s'})=0.
\]
For regular modules $H_{r_1}$ and $H_{r_2}$ in different tubes, $\Ext(H_{r_i},H_{r_j})=\Hom(H_{r_i},H_{r_j})=0,$ for $i\not=j.$
\end{lemma}

\begin{definition}\label{deff}
For $\alpha,m$ as above, define
\[
X(\alpha,m):=\left\{x\in \bfE_{\bfV}\ \middle|\ 
\begin{aligned}
&M_x\cong P_{\alpha}\oplus H_r\oplus I_{\alpha},\\
& H_r\text{ is a direct sum of indecomposable regular modules}
\end{aligned}\right\}.
\]
This is an irreducible locally closed subvariety of $\bfE_{\bfV}.$ Denote by $j_{\alpha,m}:X(\alpha,m)\to \bfE_{\bfV}$ the natural inclusion.
\end{definition}

Following \cite{XIAO2023510}, we may also define an order on the set $R$: for $\alpha,\beta\in R$, we write $\alpha>\beta$ if there exist $l>0$ and $l'>0$ such that $\alpha_s^P=\beta_s^P$ for $1\le s<l$, $\alpha_l^P>\beta_l^P$, and $\alpha_s^I=\beta_s^I$ for $1\le s<l'$, $\alpha_{l'}^I>\beta_{l'}^I$ or $\alpha_s^P=\beta_s^P$ for $1\le s<l$, $\alpha_l^P>\beta_l^P$ and $\alpha_s^I=\beta_s^I$ for any $s,$ or $\alpha_s^I=\beta_s^I$ for $1\le s<l'$, $\alpha_{l'}^I>\beta_{l'}^I$ and $\alpha_s^P=\beta_s^P$ for any $s.$

\begin{definition}
For an $I$-graded space $\bfV$, consider the set of pairs
\[
\{(\alpha,m)\mid (m,m)+\sum_{s\ge 1}\alpha_s^P(s-1,s)+\alpha_s^I(s,s-1)=\dim \bfV\}.
\]
We define an order on this set by
\[
(\alpha',m')\le (\alpha,m)\quad \text{if and only if}\quad \alpha'\ge \alpha.
\]
\end{definition}

\begin{lemma}\cite[Proposition 7.3]{XIAO2023510}\label{stra}
In $\bfE_{\bfV}$, we have
\[
\overline{X(\alpha,m)}\subset \bigcup_{(\alpha',m')\le (\alpha,m)} X(\alpha',m').
\]
\end{lemma}

This lemma implies that $X(\alpha,m)$ is open in its closure.
\section{Flag sheaf complexes}

In this section, we recall Lusztig's geometric construction of flag sheaf complexes in the Kronecker case \cite{Lusztig1992} and give explicit descriptions of certain flag sheaf complexes on $X(\alpha,m)$.

Since $\bfE_{\bfV}$ has a natural $\bbF_p$-structure with Frobenius map $\operatorname{Fr}$, Lusztig also defines $\cD_{m,\bfG_{\bfV}}^{b}(\bfE_{\bfV})$ as the derived category of mixed $\bfG_{\bfV}$-equivariant sheaf complexes $(\cD_{m,r,\bfG_{\bfV}}^{b}(\bfE_{\bfV}) \text{ with Frobenius map } \operatorname{Fr}^r \text{ if } q=p^r\bigr).$ The mixed structure is taken with respect to the $\bbF_q$-structure of the variety, which gives rise to a Weil structure on $\cD_{\bfG_{\bfV}}^{b}(\bfE_{\bfV}),$ the derived category of $\bfG_{\bfV}$-equivariant sheaf complexes. We will review the induction functor defined by Lusztig.

\[
\Ind_{\nu',\nu''}^{\nu}:
\cD_{m,\bfG_{\bfV'}}^{b}(\bfE_{\bfV'})\times
\cD_{m,\bfG_{\bfV''}}^{b}(\bfE_{\bfV''})
\longrightarrow
\cD_{m,\bfG_{\bfV}}^{b}(\bfE_{\bfV}),
\]
where the dimension vectors of the $I$-graded vector spaces $\bfV$, $\bfV'$, $\bfV''$ are $\nu$, $\nu'$, $\nu''$, respectively. Let
\[
\bfE_{\bfV}'=
\left\{(x,\bfW,\rho_1,\rho_2)\ \middle|\ 
\begin{aligned}
&x\in \bfE_{\bfV},\ \bfW\cong \bfV'' \text{ as an $I$-graded space},\\
&\rho_1:\bfV/\bfW\cong \bfV',\ \rho_2:\bfW\cong \bfV''
\end{aligned}\right\},
\]
\[
\bfE_{\bfV}''=
\left\{(x,\bfW)\ \middle|\ x\in \bfE_{\bfV},\ \bfW\cong \bfV'' \text{ as an $I$-graded space}\right\},
\]
with maps defined by
\[
p_1(x,\bfW,\rho_1,\rho_2)
=
\bigl(\rho_1(x|_{\bfV/\bfW})\rho_1^{-1},\ \rho_2(x|_{\bfW})\rho_2^{-1}\bigr),
\qquad
p_2(x,\bfW,\rho_1,\rho_2)=(x,\bfW),
\qquad
p_3(x,\bfW)=x.
\]
\begin{equation}\label{eq111}
 \begin{tikzcd}
\bfE_{\bfV'}\times \bfE_{\bfV''} & \bfE_{\bfV}' \arrow[l, "p_1"'] \arrow[r, "p_2"] & \bfE_{\bfV}'' \arrow[r, "p_3"] & \bfE_{\bfV}
\end{tikzcd}
\end{equation}

We denote the dimensions of the fibers of the smooth morphisms $p_1$ and $p_2$ by $d_1$ and $d_2$, respectively. The induction functor is defined as
\[
\Ind_{\nu',\nu''}^{\nu}
:=
{p_3}_! \, {p_2}_{\flat} \, p_1^* [d_1-d_2]\left(\frac{d_1-d_2}{2}\right).
\]
Here ${p_2}_{\flat}$ is the inverse functor of $p_2^*.$
Lusztig then defines flag sheaf complexes as follows.
For $\nu\in \bbN I$ with $\dim \bfV=\nu$, the set of flags $S_{\nu}$ is defined by
\[
S_{\nu}=\{\,(\nu_1,\nu_2\cdots \nu_n) \mid\nu_j=l_ji_j,\ i_j\in I,\ \sum_{j=1}^n l_ji_j=\nu\,\}.
\]

For $\underline{\nu}=(\nu_j)\in S_{\nu}$, we consider the flag variety $\widetilde{\mathcal{F}_{\underline{\nu}}}$ defined by
\[
\left\{(x,f)\ \middle|\ 
\begin{aligned}
&x\in \bfE_{\bfV},\ f=(0=V_n\subset V_{n-1}\subset \cdots \subset V_1\subset V_0=\bfV)\text{ is a filtration of }\bbN I\text{-graded spaces},\\
&x(V_i)\subset V_i,\ \dim (V_{i-1}/V_i)=\nu_i
\end{aligned}\right\}.
\]
As shown in \cite[Chapter 9]{lusztig2010introduction}, this is a smooth irreducible variety; denote its dimension by $f(\underline{\nu}).$

There is a natural projective morphism $\pi_{\underline{\nu}}:\widetilde{\mathcal{F}_{\underline{\nu}}}\to \bfE_{\bfV}$. The flag sheaf complex corresponding to $\underline{\nu}$ is $$\mathcal{L}_{\underline{\nu}}:={\pi_{\underline{\nu}}}_!(\overline{\bbQ_l})[f(\underline{\nu})](\frac{f(\underline{\nu})}{2}).$$ By \cite{BBD}, it is semisimple. 

Now we define the Grothendieck group of $\cD_{m,r,\bfG_{\bfV}}^{b}(\bfE_{\bfV})$.
\begin{definition}
The Grothendieck group $\mk_{\bfV}^r$ of $\cD_{m,r,\bfG_{\bfV}}^{b}(\bfE_{\bfV})$ for $q=p^r$ is defined as follows.
\begin{itemize}
\item[$\bullet$] If $A\cong A'$ in $\cD_{m,r,\bfG_{\bfV}}^{b}(\bfE_{\bfV})$, then they have the same image in $\mk_{\bfV}^r$, denoted by $[A]$.
\item[$\bullet$] If $A\xrightarrow{f} B\xrightarrow{g} C\xrightarrow{+1}$ is a distinguished triangle in $\mathcal{D}^b_{G_{\mathbf{V}}}(\mathbf{E}_{\mathbf{V}})$ compatible with the Weil structure, then $[B]=[A]+[C]$ in $\mk_{\bfV}^r$.
\item[$\bullet$] For $k\in \bbZ$, we have $[A\otimes \overline{\mathbb{Q}_l}(k)]=q^{-k}[A]$, where $\overline{\mathbb{Q}_l}(k)$ denote the Tate twist.
\end{itemize}
\end{definition}
Induction functor $\Ind$ induces a multiplication $\cdot$ on the graded direct sum
\[
\mk^r=\bigoplus_{\bfV} \mk_{\bfV}^r.
\]

Let $X^{\operatorname{Fr}^r}$ and $G^{\operatorname{Fr}^r}$ denote the fixed point sets of $X$ and $G$ under their Frobenius maps ${\operatorname{Fr}^r}.$ Respectively, we denote by $\tilde{\cH}_{G^{\operatorname{Fr}^r}}(X^{\operatorname{Fr}^r})$ the $\bbC$-vector space of $G^{\operatorname{Fr}^r}$-invariant functions on $X^{\operatorname{Fr}^r}$. For any $G$-equivariant morphism $f:X\rightarrow Y$ which respects the $\bbF_q$-structures, we denote its restriction to $X^{\operatorname{Fr}^r}$ by the same letter $f:X^{\operatorname{Fr}^r}\rightarrow Y^{\operatorname{Fr}^r}$, and there are $\bbC$-linear maps
\begin{align*}
f^*:\tilde{\cH}_{G^{\operatorname{Fr}^r}}(Y^{\operatorname{Fr}^r})&\rightarrow \tilde{\cH}_{G^{\operatorname{Fr}^r}}(X^{\operatorname{Fr}^r}) &f_!: \tilde{\cH}_{G^{\operatorname{Fr}^r}}(X^{\operatorname{Fr}^r})&\rightarrow \tilde{\cH}_{G^{\operatorname{Fr}^r}}(Y^{\operatorname{Fr}^r})\\
\varphi&\mapsto (x\mapsto f(\varphi(x))), & \psi&\mapsto(y\mapsto \sum_{x\in f^{-1}(y)}\psi(x)).
\end{align*}

For any mixed Weil complex $A\in \cD^b_{G,r,m}(X)$ with the Weil structure $\xi:{\operatorname{Fr}^r}^*(A)\rightarrow A$ and $x\in X^{\operatorname{Fr}^r}, s\in \bbZ$, there is an isomorphism $H^s(\xi)_x:H^s(A)_x\rightarrow H^s(A)_x$ of the stalk at $x$ of the $s$-th cohomology sheaf. Taking the alternating sum of the traces of these isomorphisms, we obtain a value 
$$\chi(A)(x)=\sum_{s\in \bbZ}(-1)^s\tr(H^s(\xi)_x)\in\overline{\bbQ}_l\cong \bbC$$
and a function $\chi(A)\in \tilde{\cH}_{G^{\operatorname{Fr}^r}}(X^{\operatorname{Fr}^r})$. Moreover, $\chi({-})$ induces a map from the Grothendieck group of $\cD^b_{G,r,m}(X)$ to $\tilde{\cH}_{G^F}(X^F)$, see \cite[Lemma 5.3.12]{Pramod-2021}, which is called the trace map.

The multiplication on $\oplus_{\bfV}\cH_{\bfG_{\bfV}^{\operatorname{Fr}^r}}(\bfE_{\bfV}^{\operatorname{Fr}^r})$ is defined as follows.
\begin{definition}
 The multiplication $\cH_{\bfG_{\bfV'}^{\operatorname{Fr}^r}}(\bfE_{\bfV'}^{\operatorname{Fr}^r})\otimes \cH_{\bfG_{\bfV''}^{\operatorname{Fr}^r}}(\bfE_{\bfV''}^{\operatorname{Fr}^r})\rightarrow \cH_{\bfG_{\bfV}^{\operatorname{Fr}^r}}(\bfE_{\bfV}^{\operatorname{Fr}^r})$ is defined by $$fg=\frac{(-\sqrt{q})^{-\sum_{i\in I} \nu'_i\nu''_i-\sum_{h\in H}\nu'_{s(h)}\nu''_{t(h)}}}{|\bfG_{\bfV'}^{\operatorname{Fr}^r}\times \bfG_{\bfV''}^{\operatorname{Fr}^r}|}{p_3}_!{p_2}_!p_1^*(f\otimes g),$$ with $p_1,$ $p_2,$ $p_3$ in Equation~\ref{eq111}.
\end{definition}
\begin{theorem}[{\cite[Theorem 5.3.13]{Pramod-2021}}]\label{sheaf-function correspondence}
For any $G$-equivariant morphism $f:X\rightarrow Y$ which respects the $\bbF_q$-structures and $A\in \cD^b_{G,r,m}(X), B\in \cD^b_{G,r,m}(Y)$, we have
\begin{align*}
&\chi({A[n]})=(-1)^n\chi(A),\ \chi({A(\frac{n}{2})})=\sqrt{q}^{-n}\chi(A),\ \chi({A\boxtimes B})=\chi(A)\otimes \chi(B),\\
&\chi({f^*(B)})=f^*(\chi(B)),\ \chi({f_!(A)})=f_!(\chi(A)).
\end{align*}
\end{theorem}

Fix $\alpha\in R$ and an $I$-graded space $\bfV$ such that
\[
\dim \bfV=(m,m)+\sum_{s\ge 1}\alpha_s^P(s-1,s)+\alpha_s^I(s,s-1).
\]
Denote by $X(\alpha)=X(\alpha,(1,\cdots,1))$ the locally closed subspace
\[
\left\{x\in \bfE_{\bfV}\ \middle|\ 
\begin{aligned}
&M_x\cong P_{\alpha}\oplus H_r\oplus I_{\alpha},\\
&H_r\cong H_{r_1}\oplus\cdots\oplus H_{r_m},\\
&H_{r_i}\text{ is regular},\ \dim H_{r_i}=(1,1),\ H_{r_i}\not\cong H_{r_i}\ (i\ne j)
\end{aligned}\right\},
\]
and denote the embedding map by $j_{\alpha}:X(\alpha)\to \bfE_{\bfV}$. By \cite[Proposition 4.14(a)]{Lusztig1992}, $X(\alpha)$ is open, dense and smooth in its closure $\overline{X(\alpha)}.$ Moreover, by \cite[6.10]{Lusztig1992}, $X(\alpha,m)\subset\overline{X(\alpha)}.$ Thus $X(\alpha)\subset X(\alpha,m)\subset \overline{X(\alpha,m)}=\overline{X(\alpha)},$ and $X(\alpha)$ is an open subset of $X(\alpha,m)$ and $X(\alpha,m)$ is an open subset of $\overline{X(\alpha)}.$

In~\cite[6.14]{Lusztig1992}, for
\[
\alpha=\sum_{i=1}^r \alpha_{s_i}^P P_{s_i}+\sum_{j=1}^t \alpha_{s'_j}^I I_{s'_j},
\]
with $\alpha_{s_i}^P\ne 0$ and $\alpha_{s'_j}^I\ne 0$, and with
$s_{i'}\le s_{i}$ and $s'_{i'}\ge s'_i$ for $i'\le i$,
we consider the following complexes of sheaves.

First, we denote by $\bfV_{s_i}$ the $I$-graded space of dimension $\alpha_{s_i}^P(s_i-1,s_i)$, and by $\bfV'_{s'_j}$ the $I$-graded space of dimension $\alpha_{s'_j}^I(s'_j-1,s'_j)$. For a sequence $\underline{\lambda}=(\lambda_1,\ldots,\lambda_k)$ with $\sum_{i=1}^k \lambda_i=m$ and $\lambda_1\ge \lambda_2\ge \cdots \ge \lambda_k$, we denote by $\bfV_{\lambda_i}$ the $I$-graded space of dimension $(\lambda_i,\lambda_i)$.

If we write $\Ind_{\nu',\nu''}^{\nu}(\mathcal{F}\boxtimes \mathcal{G})$ as $\mathcal{F}\diamond \mathcal{G}$, then
\begin{equation}\label{alpha}
\mathcal{L}_{\alpha,\underline{\lambda}}:=\overline{\bbQ_l}|_{\bfE_{\bfV_{s_1}}}\diamond \cdots \diamond \overline{\bbQ_l}|_{\bfE_{\bfV_{s_r}}}
\diamond \overline{\bbQ_l}|_{\bfE_{\bfV_{\lambda_1}}}\diamond \overline{\bbQ_l}|_{\bfE_{\bfV_{\lambda_2}}}\cdots \diamond \overline{\bbQ_l}|_{\bfE_{\bfV_{\lambda_k}}}
\diamond \overline{\bbQ_l}|_{\bfE_{\bfV'_{s'_1}}}\diamond \cdots \diamond \overline{\bbQ_l}|_{\bfE_{\bfV'_{s'_t}}}
\end{equation}
is semisimple and supported on $\overline{X(\alpha,m)}$. Indeed, for every geometric point $i_x: \operatorname{Spec(\overline{\bbF_q})}\rightarrow \overline{X(\alpha,m)},$ $i_x^*\mathcal{L}_{\alpha,\underline{\lambda}}\not=0,$ because there exists $n\gg0$ such that after applying the trace map $\chi^n$ of $\operatorname{Fr_{q^n}}$, $\chi^n(\mathcal{L}_{\alpha,\underline{\lambda}})(x)\not=0$.

We denote by $f_{\bfV}$ the $\bbZ[\sqrt{q},\sqrt{q}^{-1}]$-subspace of $\mk_{\bfV}^r$ spanned by $[IC(X(\alpha),L_{\chi}),f_{\alpha,\chi}]$, where
\[
IC(X(\alpha),L_{\chi})={j_{\alpha}}_{!*}\bigl(L_{\chi}[\dim X(\alpha)](\frac{\dim X(\alpha)}{2})\bigr),
\]
\[
\dim \bfV=\nu=(m,m)+\sum_{s\ge 1}\alpha_s^P(s-1,s)+\alpha_s^I(s,s-1),
\]
 $\chi\in X(S_m)$ determines a local system $L_{\chi}$ on $X(\alpha)$ and the Weil structure $f_{\alpha,\chi}$ will be defined in Section \ref{4sec}. Since $\mathcal{L}_{\alpha,\underline{\lambda}}$ can be regarded as a flag sheaf $\mathcal{L}_{\underline{\nu}},$ and in this case $\pi_{\underline{\nu}}:\widetilde{\mathcal{F}_{\underline{\nu}}}\rightarrow \bfE_{\bfV}$ can be decomposed as $\widetilde{\mathcal{F}_{\underline{\nu}}}\rightarrow \overline{X(\alpha)}\rightarrow \bfE_{\bfV}$. The restriction of $\pi_{\underline{\nu}}$ to $\pi_{\underline{\nu}}^{-1}(X(\alpha))$ is a covering map with $S_m$-set $S_m/S_{\lambda_1}\times \cdots\times S_{\lambda_k}.$

 Thus \begin{equation}\label{eq121}
 \mathcal{L}_{\alpha,\underline{\lambda}}=\oplus_{\chi\in X(S_m)}IC(X(\alpha),L_{\chi})^{\oplus K_{\underline{\lambda}},\chi}\oplus L',
 \end{equation} where $\supp L'\cap X(\alpha)=\emptyset.$

For any representation $M$, we denote its $\bfG_{\bfV}$-orbit by $\cO_M$.

In what follows, we give explicit descriptions of $\mathcal{L}_{\alpha,\underline{\lambda}}$ on $X(\alpha,m)$.

\begin{proposition}\label{suan}
With the notation above, we have the following equality:
\[
\begin{aligned}
&j_{\alpha,m}^*\Bigl(
\overline{\bbQ_l}|_{\bfE_{\bfV_{s_1}}}\diamond \cdots \diamond \overline{\bbQ_l}|_{\bfE_{\bfV_{s_r}}}
\diamond \overline{\bbQ_l}|_{\bfE_{\bfV_{\lambda_1}}}\diamond \overline{\bbQ_l}|_{\bfE_{\bfV_{\lambda_2}}}\cdots \diamond \overline{\bbQ_l}|_{\bfE_{\bfV_{\lambda_k}}}
\diamond \overline{\bbQ_l}|_{\bfE_{\bfV'_{s'_1}}}\diamond \cdots \diamond \overline{\bbQ_l}|_{\bfE_{\bfV'_{s'_t}}}
\Bigr)\\
={}&\Bigl(
\overline{\bbQ_l}|_{\cO_{P^{\oplus \alpha_{s_1}^P}_{s_1}}}\diamond \cdots \diamond \overline{\bbQ_l}|_{\cO_{P^{\oplus \alpha_{s_r}^P}_{s_r}}}
\diamond \overline{\bbQ_l}|_{X(0,\lambda_1)}\diamond \overline{\bbQ_l}|_{X(0,\lambda_2)}\cdots \diamond \overline{\bbQ_l}|_{X(0,\lambda_k)}
\diamond \overline{\bbQ_l}|_{\cO_{I^{\oplus \alpha_{{s'}_1}^I}_{{s'}_1}}}\diamond \cdots \diamond \overline{\bbQ_l}|_{\cO_{I^{\oplus \alpha_{{s'}_t}^I}_{{s'}_t}}}
\Bigr).
\end{aligned}
\]
\end{proposition}

\begin{proof}
Since $\cO_{P^{\oplus \alpha_{s_k}^P}_{s_k}}$ is an open dense subvariety of $\bfE_{\bfV_{s_k}}$ and $\cO_{I^{\oplus \alpha_{s'_k}^I}_{s'_k}}$ is an open dense subvariety of $\bfE_{\bfV_{s'_k}}$, Lemma~\ref{stra} implies that any
\[
x\in \bfE_{\bfV_{s_k}}\setminus \cO_{P^{\oplus \alpha_{s_k}^P}_{s_k}},
\quad M_x\ \text{has a direct summand }P_{\tilde{s_k}}\text{ with }\tilde{s_k}<s_k,
\]
and any
\[
x\in \bfE_{\bfV_{s'_k}}\setminus \cO_{I^{\oplus \alpha_{s'_k}^I}_{s'_k}},
\quad M_x\ \text{has a direct summand }I_{\tilde{s'_k}}\text{ with }\tilde{s'_k}<s'_k.
\]

By Lemma~\ref{le} and induction, we obtain
\[
\begin{aligned}
&j_{\alpha,m}^*\Bigl(
\overline{\bbQ_l}|_{\bfE_{\bfV_{s_1}}}\diamond \cdots \diamond \overline{\bbQ_l}|_{\bfE_{\bfV_{s_r}}}
\diamond \overline{\bbQ_l}|_{\bfE_{\bfV_{\lambda_1}}}\diamond \overline{\bbQ_l}|_{\bfE_{\bfV_{\lambda_2}}}\cdots \diamond \overline{\bbQ_l}|_{\bfE_{\bfV_{\lambda_k}}}
\diamond \overline{\bbQ_l}|_{\bfE_{\bfV'_{s'_1}}}\diamond \cdots \diamond \overline{\bbQ_l}|_{\bfE_{\bfV'_{s'_t}}}
\Bigr)\\
={}&j_{\alpha,m}^*\Bigl(
\overline{\bbQ_l}|_{\cO_{P^{\oplus \alpha_{s_1}^P}_{s_1}}}\diamond \cdots \diamond \overline{\bbQ_l}|_{\cO_{P^{\oplus \alpha_{s_r}^P}_{s_r}}}
\diamond \overline{\bbQ_l}|_{\bfE_{\bfV_{\lambda_1}}}\diamond \overline{\bbQ_l}|_{\bfE_{\bfV_{\lambda_2}}}\cdots \diamond \overline{\bbQ_l}|_{\bfE_{\bfV_{\lambda_k}}}
\diamond \overline{\bbQ_l}|_{\cO_{I^{\oplus \alpha_{{s'}_1}^I}_{{s'}_1}}}\diamond \cdots \diamond \overline{\bbQ_l}|_{\cO_{I^{\oplus \alpha_{{s'}_t}^I}_{{s'}_t}}}
\Bigr).
\end{aligned}
\]

Moreover, since $\overline{X(0,m)}=\bfE_{\bfV_m}$, for any point $x\in \bfE_{\bfV_m}\setminus X(0,m),$ $M_x$ has a preprojective and a preinjective direct summand. By a direct dimension calculation and Lemma~\ref{le}, if $M$ is a regular representation and $V$ is a subrepresentation (or a quotient) of $M$ with $\dim V=(u,u)$, then $V$ is regular.

Therefore,
\[
\begin{aligned}
&j_{\alpha,m}^*\Bigl(
\overline{\bbQ_l}|_{\cO_{P^{\oplus \alpha_{s_1}^P}_{s_1}}}\diamond \cdots \diamond \overline{\bbQ_l}|_{\cO_{P^{\oplus \alpha_{s_r}^P}_{s_r}}}
\diamond \overline{\bbQ_l}|_{\bfE_{\bfV_{\lambda_1}}}\diamond \overline{\bbQ_l}|_{\bfE_{\bfV_{\lambda_2}}}\cdots \diamond \overline{\bbQ_l}|_{\bfE_{\bfV_{\lambda_k}}}
\diamond \overline{\bbQ_l}|_{\cO_{I^{\oplus \alpha_{{s'}_1}^I}_{{s'}_1}}}\diamond \cdots \diamond \overline{\bbQ_l}|_{\cO_{I^{\oplus \alpha_{{s'}_t}^I}_{{s'}_t}}}
\Bigr)\\
={}&\Bigl(
\overline{\bbQ_l}|_{\cO_{P^{\oplus \alpha_{s_1}^P}_{s_1}}}\diamond \cdots \diamond \overline{\bbQ_l}|_{\cO_{P^{\oplus \alpha_{s_r}^P}_{s_r}}}
\diamond \overline{\bbQ_l}|_{X(0,\lambda_1)}\diamond \overline{\bbQ_l}|_{X(0,\lambda_2)}\cdots \diamond \overline{\bbQ_l}|_{X(0,\lambda_k)}
\diamond \overline{\bbQ_l}|_{\cO_{I^{\oplus \alpha_{{s'}_1}^I}_{{s'}_1}}}\diamond \cdots \diamond \overline{\bbQ_l}|_{\cO_{I^{\oplus \alpha_{{s'}_t}^I}_{{s'}_t}}}
\Bigr).
\end{aligned}
\]
\end{proof}

Let $\underline{\lambda}=(\lambda_1,\cdots,\lambda_n)$, with $\sum_{i=1}^n\lambda_i=m.$
Define $W_{\alpha,\underline{\lambda}}$ to be the variety
\[
\left\{(x,f)\ \middle|\ 
\begin{aligned}
&x\in X(\alpha,m),\\
&f=(0=V_0^0\subset V_0^1\cdots\subset V_0^s=V_0\subset V_1\subset \cdots \subset V_n=V_n^{t+1}\subset V_n^t\subset\cdots\subset V_n^1 \subset V_{n}^0=\bfV)\\&\text{ is a filtration of }\bbN I\text{-graded spaces},\\
&x(V_i)\subset V_i,\ \dim V_i/V_{i-1}=(\lambda_i,\lambda_i), x(V_0^j)\subset V_0^j,x(V_n^j)\subset V_n^j,\\
&M_{x|_{V_i/V_{i-1}}}\text{ is a direct sum of indecomposable regular modules},\ i=1,\ldots,n,\\
&M_{x|_{V_0^{i+1}/V_0^i}}\cong I^{\oplus \alpha_{{s'}_i}^I}_{s'_i},\qquad M_{x|_{V_{n}^{i-1}/V^i_n}}\cong P^{\oplus \alpha_{s_i}^P}_{s_i}
\end{aligned}\right\}.
\]
We also denote by $p_{\alpha,\underline{\lambda}}:W_{\alpha,\underline{\lambda}}\to X(\alpha,m),$ the restriction of the flag variety $\widetilde{\mf_{\underline{\nu}}}\rightarrow \overline{X(\alpha)}$ corresponding to $\ml_{\alpha,\underline{\lambda}}$ in Equation~\ref{alpha} on $X(\alpha,m).$ This map is proper. For $W_{\alpha,\underline{\lambda}}^{rss}={p_{\alpha,\underline{\lambda}}}^{-1}X(\alpha),$ $p_{\alpha,\underline{\lambda}}^{rss}$ denote the restriction of $p_{\alpha,\underline{\lambda}}$ on $W_{\alpha,\underline{\lambda}}^{rss}.$ It is a covering map. Then $${p_{\alpha,\underline{\lambda}}^{rss}}_!(\overline{\bbQ_l})=\oplus_{\chi\in X(S_m)} L_{\chi}^{\oplus K_{\underline{\lambda},\chi}}.$$
\begin{lemma}\label{lem24}
 If we denote $X(\alpha)\xrightarrow{j_{\alpha}^0} X(\alpha,m),$ and $K_{\underline{\lambda},\chi}$ is a Kostka number, \begin{equation}\label{22}
 {p_{\alpha,\underline{\lambda}}}_!(\overline{\bbQ_l})[\dim X(\alpha)](\frac{\dim(X(\alpha))}{2})=\oplus_{\chi\in X(S_m)} {j_{\alpha}^0}_{!*}L_{\chi}[\dim X(\alpha)](\frac{\dim X(\alpha)}{2})^{\oplus K_{\underline{\lambda},\chi}}.
\end{equation}Furthermore, we have
 $j_{\alpha,m}^*IC(X(\alpha),L_{\chi})\cong {j_{\alpha}^0}_{!*}L_{\chi}[\dim X(\alpha)](\frac{\dim X(\alpha)}{2}).$
\end{lemma}
\begin{proof}
First, let $P=(l_1,\cdots,l_t),$ $\sum l_i=m,$ and $\underline{\mu_i}=(\mu_i^1,\cdots,\mu_i^{u_i}),$ such that $\sum\mu^j_i=l_i,$ $\mu_i^1\geq \cdots\geq \mu_i^{u_i}.$ Denote $\bfV_{l_i}$ as the graded space of dimension $(l_i,l_i).$ Define the subset $$X_{P,(\underline{\mu_i}_{i=1,\cdots,t})}=\left\{x\in X(\alpha,m)\middle|\begin{aligned}
 &M_x\cong P_{\alpha}\oplus I_{\alpha}\oplus H_r,H_r\cong H_{r_1}\oplus\cdots\oplus H_{r_t}, H_{r_i}\cong \oplus_j H_{r_i^j},\\& H_{r_i^j}\text{ is indecomposable of dimension }(\mu_i^j,\mu_i^j),\\
 & H_{r_i}\text{ and }H_{r_j}\text{ are in different tubes for }i\not=j.
\end{aligned}\right\}.$$ It is a locally closed subvariety of $X(\alpha,m)$ \cite{wreo6318}.
In this way, we denote $X(\alpha)$ as $X_{P_0,(\underline{{\mu_0}_i})}.$

It is straightforward to see that $X(\alpha,m)=\bigsqcup_{(P,(\underline{\mu_i}_{i=1,\cdots,t}))\in \Theta} X_{P,(\underline{\mu_i}_{i=1,\cdots,t})},$ which is a finite union with index set $\Theta$. Moreover, $\dim X_{P,(\underline{\mu_i}_{i=1,\cdots,t})}=\dim \mo_{x}+t,$ $x\in X_{P,(\underline{\mu_i}_{i=1,\cdots,t})}.$ We also have that $\dim W_{\alpha,\underline{\lambda}}=\dim X(\alpha)=m+\dim \mo_{x_0},$ for any $x_0\in X(\alpha).$

For $x\in X_{P,(\underline{\mu_i}_{i=1,\cdots,t})},\ x_0\in X(\alpha)$ $$\dim \mo_{x_0}-\dim \mo_{x}=\dim \Hom_{kQ}(M_x,M_x)-\dim \Hom_{kQ}(M_{x_0},M_{x_0})=\sum_{i=1}^t\dim \Hom_{kQ}(H_{r_i},H_{r_i})-m.$$

 We consider the case $\underline{\lambda}=(1,\cdots,1).$ For $x\in X_{P,(\underline{\mu_i}_{i=1,\cdots,t})},$ $\dim p_{\alpha,\underline{\lambda}}^{-1}(x)=\sum_{i=1}^t d_i,$ where $d_i$ is the dimension of Springer fiber of $x_i,$ where $M_{x_i}\cong \oplus_{j=1}^{u_i} H_{{r'}_i^j},$ and $H_{{r'}_i^j}$ is an indecomposable nilpotent representation of dimension $\mu_i^j$ of Jordan quiver $Q_J$.

 By \cite[Theorem 4.6]{Steinberg1976}, Springer resolution is semismall, thus $d_i\leq \frac{1}{2}(l_i(l_i-1)-\dim \mo_{x_i}).$ 

 Then $$\begin{aligned}
 \dim p_{\alpha,\underline{\lambda}}^{-1}(x)=\sum_{i=1}^t d_i&\leq \sum_{i=1}^t \frac{1}{2}(l_i(l_i-1)-\dim \mo_{x_i})\\
 &=\sum_{i=1}^t \frac{1}{2}(l_i(l_i-1)-\dim \bfG l_{l_i}+\dim \Hom_{kQ_J}(M_{x_i},M_{x_i}))\\
 &=\sum_{i=1}^t \frac{1}{2}(l_i(l_i-1)-\dim \bfG l_{l_i}+\dim \Hom_{kQ}(M_{r_i},M_{r_i}))\\
 &=\frac{1}{2}(\dim W_{\alpha,\underline{\lambda}}-\dim X_{P,(\underline{\mu_i}_{i=1,\cdots,t})}+t)+\sum_{i=1}^t \frac{1}{2}(l_i(l_i-1)-\dim \bfG l_{l_i})\\
 &=\frac{1}{2}(\dim W_{\alpha,\underline{\lambda}}-\dim X_{P,(\underline{\mu_i}_{i=1,\cdots,t})}+t-m).
 \end{aligned}$$

 Thus when $t\not=m,$ for $x\in X_{P,(\underline{\mu_i}_{i=1,\cdots,t})},$ $ \dim p_{\alpha,\underline{\lambda}}^{-1}(x)<\frac{1}{2}(\dim W_{\alpha,\underline{\lambda}}-\dim X_{P,(\underline{\mu_i}_{i=1,\cdots,t})})$ and when $t=m,$ $X_{P,(\underline{\mu_i}_{i=1,\cdots,t})}=X(\alpha),$ for $x\in X(\alpha),$ $p_{\alpha,\underline{\lambda}}^{-1}(x)$ is a finite set. 
 
 For other $\underline{\lambda},$ we have that $p^{-1}_{\alpha,(1,\cdots,1)}(x)\rightarrow p^{-1}_{\alpha,\underline{\lambda}}(x)$ is a surjective map. Thus when $t\not=m,$ for $x\in X_{P,(\underline{\mu_i}_{i=1,\cdots,t})},$ $ \dim p_{\alpha,\underline{\lambda}}^{-1}(x)<\frac{1}{2}(\dim W_{\alpha,\underline{\lambda}}-\dim X_{P,(\underline{\mu_i}_{i=1,\cdots,t})})$ and when $t=m,$ $X_{P,(\underline{\mu_i}_{i=1,\cdots,t})})=X(\alpha),$ for $x\in X(\alpha),$ $p_{\alpha,\underline{\lambda}}^{-1}(x)$ is a finite set. 
 
 Since $W_{\alpha,\underline{\lambda}}$ is open in the flag variety, it is smooth and irreducible. Denote $L=\overline{\bbQ_l}[\dim X(\alpha)](\frac{\dim X(\alpha)}{2}).$
 
 We now consider $Z=X(\alpha,m)-X(\alpha)=\bigsqcup_{(P,(\underline{\mu_i}_{i=1,\cdots,t}))\in \Theta-(P_0,(\underline{{\mu_0}_i}))} X_{P,(\underline{\mu_i}_{i=1,\cdots,t})},$ and $i:Z\rightarrow X(\alpha,m).$ We will prove that $i^*(p_{\alpha,\underline{\lambda}})!L\in {}^p\cD^b_{\bfG_{\bfV},c}(Z,\overline{\bbQ_l})^{\leq -1}$ and $i^!(p_{\alpha,\underline{\lambda}})!L\in {}^p\cD^b_{\bfG_{\bfV},c}(Z,\overline{\bbQ_l})^{\geq 1}.$

 As in the proof in \cite[Lemma 3.8.3]{Pramod-2021}, $$\dim \operatorname{supp}L=\dim X(\alpha).$$ Since for any $y\in X_{P,(\underline{\mu_i}_{i=1,\cdots,t})},$ $$\operatorname{H}^j({p_{\alpha,\underline{\lambda}}}_!L)\cong \mathbf{H}^j_c({p_{\alpha,\underline{\lambda}}}^{-1}(y),L|_{{p_{\alpha,\underline{\lambda}}}^{-1}(y)}),$$ and $\dim {p_{\alpha,\underline{\lambda}}}^{-1}(y)<\frac{1}{2}(\dim X(\alpha)-\dim X_{P,(\underline{\mu_i}_{i=1,\cdots,t})}).$ By \cite[Theorem 2.7.4]{Pramod-2021}, $\operatorname{H}^j({p_{\alpha,\underline{\lambda}}}_!L)=0$ unless $j< -\dim X_{P,(\underline{\mu_i}_{i=1,\cdots,t})}.$ Thus $\operatorname{supp} \operatorname{H}^j(i^*{p_{\alpha,\underline{\lambda}}}_!L)\subset\cup X_{P,(\underline{\mu_i}_{i=1,\cdots,t})},$ where $(P,(\underline{\mu_i}_{i=1,\cdots,t}))\in \Theta-(P_0,(\underline{{\mu_0}_i}))$ and $j<- \dim X_{P,(\underline{\mu_i}_{i=1,\cdots,t})}.$ Then $\dim \operatorname{supp}\operatorname{H}^j(i^*{p_{\alpha,\underline{\lambda}}}_!L)<-j,$ hence $$i^*{p_{\alpha,\underline{\lambda}}}_!L\in {}^p\cD^b_{\bfG_{\bfV},c}(Z,\overline{\bbQ_l})^{\leq -1}.$$ 

 To prove $i^!(p_{\alpha,\underline{\lambda}})!L\in {}^p\cD^b_{\bfG_{\bfV},c}(Z,\overline{\bbQ_l})^{\geq 1}$ is equivalent to prove $\mathbf{D}(i^!(p_{\alpha,\underline{\lambda}})!L)\in \mathbf{D}({}^p\cD^b_{\bfG_{\bfV},c}(Z,\overline{\bbQ_l})^{\geq 1}).$ Since $\overline{\bbQ_l}$ is a field, by \cite[Lemma 3.1.11]{Pramod-2021}, $\mathbf{D}({}^p\cD^b_{\bfG_{\bfV},c}(Z,\overline{\bbQ_l})^{\geq 1})={}^p\cD^b_{\bfG_{\bfV},c}(Z,\overline{\bbQ_l})^{\leq -1}.$ And $\mathbf{D}(i^!(p_{\alpha,\underline{\lambda}})!L)=i^*(p_{\alpha,\underline{\lambda}})!L$ because $p_{\alpha,\underline{\lambda}}$ is proper and $\mathbf{D}L=L.$ Then we prove that $$i^!(p_{\alpha,\underline{\lambda}})!L\in {}^p\cD^b_{\bfG_{\bfV},c}(Z,\overline{\bbQ_l})^{\geq 1}.$$

 By the same argument, we also have that ${p_{\alpha,\underline{\lambda}}}_!L\in \operatorname{Perv}(X(\alpha,m)).$ 

We calculate that $j_0^*{p_{\alpha,\underline{\lambda}}}_!L={p_{\alpha,\underline{\lambda}}^{rss}}_!(\overline{\bbQ_l}[\dim X(\alpha)](\frac{\dim X(\alpha)}{2}))=\oplus_{\chi\in X(S_m)}L_{\chi}[\dim X(\alpha)](\frac{\dim X(\alpha)}{2})^{\oplus K_{\underline{\lambda},\chi}}.$ Then by \cite[Lemma 3.3.4]{Pramod-2021}\cite{laszlo2009}[Lemma 6.1], $${p_{\alpha,\underline{\lambda}}}_!L=\oplus_{\chi\in X(S_m)} {j_{\alpha}^0}_{!*}L_{\chi}[\dim X(\alpha)](\frac{\dim X(\alpha)}{2})^{\oplus K_{\underline{\lambda},\chi}}.$$ 

Again by \cite{Pramod-2021}[Lemma 3.3.4]\cite{laszlo2009}[Lemma 6.1], it is easy to see that $$j_{\alpha,m}^*IC(X(\alpha),L_{\chi})\cong {j_{\alpha}^0}_{!*}L_{\chi}[\dim X(\alpha)](\frac{\dim X(\alpha)}{2}).$$
\end{proof}
\begin{remark}\label{rem25}
 By Proposition~\ref{suan} and Lemma~\ref{lem24}, for equation~\ref{eq121}, $$\begin{aligned}
 j_{\alpha,m}^*(\oplus_{\chi\in X(S_m)}j_{\alpha,m}^*IC(X(\alpha),L_{\chi})^{\oplus K_{\underline{\lambda},\chi}}\oplus L')&=j_{\alpha,m}^*\mathcal{L}_{\alpha,\underline{\lambda}}={p_{\alpha,\underline{\lambda}}}_!(\overline{\bbQ_l}[\dim X(\alpha)](\frac{\dim X(\alpha)}{2}))\\&=\oplus_{\chi\in X(S_m)} {j_{\alpha}^0}_{!*}L_{\chi}[\dim X(\alpha)](\frac{\dim X(\alpha)}{2})^{\oplus K_{\underline{\lambda},\chi}}\\&=\oplus_{\chi\in X(S_m)}j_{\alpha,m}^*IC(X(\alpha),L_{\chi})^{\oplus K_{\underline{\lambda},\chi}},
 \end{aligned}$$hence $\supp L'\cap X(\alpha,m)=\emptyset.$
\end{remark}
We denote $j_{s_i}:\mo_{P^{\oplus \alpha_{s_i}^P}_{s_i}}\rightarrow \bfE_{\bfV_{s_i}}$, $j_{s'_i}:\mo_{I^{\oplus \alpha_{s'_i}^I}_{s'_i}}\rightarrow \bfE_{\bfV_{s'_i}}$ and $j_{\lambda_i}: X(0,\lambda_i)\rightarrow \bfE_{\bfV_{\lambda_i}}.$
\begin{proposition}\label{propp}
 With the notation above, we have the following equality:
\[
\begin{aligned}
&\Bigl(
{j_{s_1}}_!j_{s_1}^*\overline{\bbQ_l}|_{\bfE_{\bfV_{s_1}}}\diamond \cdots \diamond {j_{s_r}}_!j_{s_r}^*\overline{\bbQ_l}|_{\bfE_{\bfV_{s_r}}}
\diamond {j_{m}}_!j_{m}^*IC(X(0,m),L_{\chi})
\diamond {j_{s'_1}}_!j_{s'_1}^*\overline{\bbQ_l}|_{\bfE_{\bfV'_{s'_1}}}\diamond \cdots \diamond {j_{s'_t}}_!j_{s'_t}^*\overline{\bbQ_l}|_{\bfE_{\bfV'_{s'_t}}}
\Bigr)\\
={}&{j_{\alpha,m}}_!j_{\alpha,m}^*\Bigl(
\overline{\bbQ_l}|_{\cO_{P^{\oplus \alpha_{s_1}^P}_{s_1}}}\diamond \cdots \diamond \overline{\bbQ_l}|_{\cO_{P^{\oplus \alpha_{s_r}^P}_{s_r}}}
\diamond IC(X(0,m),L_{\chi})
\diamond \overline{\bbQ_l}|_{\cO_{I^{\oplus \alpha_{{s'}_1}^I}_{{s'}_1}}}\diamond \cdots \diamond \overline{\bbQ_l}|_{\cO_{I^{\oplus \alpha_{{s'}_t}^I}_{{s'}_t}}}
\Bigr).
\end{aligned}
\]
\end{proposition}
\begin{proof}
 By induction, it suffices to prove that \[
\begin{aligned}
&\Bigl(
{j_{s_1}}_!j_{s_1}^*\overline{\bbQ_l}|_{\bfE_{\bfV_{s_1}}}\diamond {j_{\alpha-\alpha_{s_1}^PP_{s_1},m}}_!j_{\alpha-\alpha_{s_1}^PP_{s_1},m}^*\Bigl(\overline{\bbQ_l}|_{\bfE_{\bfV_{s_2}}}\diamond \overline{\bbQ_l}|_{\bfE_{\bfV_{s_r}}}
\diamond IC(X(0,m),L_{\chi})
\diamond \overline{\bbQ_l}|_{\bfE_{\bfV'_{s'_1}}}\diamond \cdots \diamond \overline{\bbQ_l}|_{\bfE_{\bfV'_{s'_t}}}\Bigr)
\Bigr)\\
={}&{j_{\alpha,m}}_!j_{\alpha,m}^*\Bigl(
\overline{\bbQ_l}|_{\cO_{P^{\oplus \alpha_{s_1}^P}_{s_1}}}\diamond \cdots \diamond \overline{\bbQ_l}|_{\cO_{P^{\oplus \alpha_{s_r}^P}_{s_r}}}
\diamond IC(X(0,m),L_{\chi})
\diamond \overline{\bbQ_l}|_{\cO_{I^{\oplus \alpha_{{s'}_1}^I}_{{s'}_1}}}\diamond \cdots \diamond \overline{\bbQ_l}|_{\cO_{I^{\oplus \alpha_{{s'}_t}^I}_{{s'}_t}}}
\Bigr).
\end{aligned}
\]
Since $$\left\{(x,W)|x\in X(\alpha,m),x(W)\subset W,\dim W=(m,m)+\sum_{s\ge 2}\alpha_s^P(s-1,s)+\sum_{s\ge 1}\alpha_s^I(s,s-1),\right\}$$ is the same as $$\left\{(x,W)|x\in X(\alpha,m),x|_{W}\in X(\alpha-\alpha^P_{s_1}P_{s_1},m),x|_{\bfV/W}\cong \alpha^P_{s_1}P_{s_1}\right\}.$$

By \[\begin{tikzcd}
\bfE_{\bfV_{s_1}}\times \bfE_{\bfV'} & \bfE'_{\bfV} \arrow[r, "p_2"] \arrow[l, "p_1"'] & \bfE_{\bfV}'' \arrow[r, "p_3"] & \bfE_{\bfV} \\
{X(\alpha^P_{s_1}P_{s_1})\times X(\alpha-\alpha^P_{s_1}P_{s_1},m)} \arrow[u, "{j_{s_1}\times j_{\alpha-\alpha^P_{s_1}P_{s_1},m}}"] & {p_2^{-1}p_3^{-1}(X(\alpha,m))} \arrow[r] \arrow[l] \arrow[u, "j'"] & {p_3^{-1}(X(\alpha,m))} \arrow[r] \arrow[u, "j''"] & {X(\alpha,m)} \arrow[u, "{j_{\alpha,m}}"]
\end{tikzcd}\]
Since all squares in this diagram are Cartesian, the claim follows. The cases in which the first factor is $IC(X(0,m),L_{\chi})$ or preinjective part, the proof is the same.
\end{proof}

\section{Flag sheaf complexes over $X(\alpha,m)$}

In this section, we give a more explicit description of any flag variety over the locally closed subvariety $X(\alpha,m)\subset \bfE_{\bfV}$. Since $\mathcal{L}_{(i,\cdots,i)}=\oplus \mathcal{L}_{li}[2n](n)^{\oplus s_n}$, we only need to consider the case $\underline{\nu}=(i_1,\cdots,i_n),$ $i_j\in I.$

We study $j_{\alpha,m}^*\mathcal{L}_{\underline{\nu}},$ and give an important property of ${}^p\mathbf{H}^i(j_{\alpha,m}^*\mathcal{L}_{\underline{\nu}}).$
\begin{theorem}\label{th}
For the variety $\widetilde{\mathcal{F}_{\underline{\nu}}}$ and the projective morphism $\pi_{\underline{\nu}}$, each simple constituent of ${}^p\mathbf{H}^j\!\left({j_{\alpha',m'}}^*{\pi_{\underline{\nu}}}_!\overline{\bbQ_l}\right)$ has the form $j_{\alpha',m'}^*IC(X(\alpha'),L_{\chi})$.
\end{theorem}
The theorem will be proved in Section \ref{sec}.
By Lemma~\ref{lem24} and \cite[Lemma~3.3.11]{Pramod-2021}, $j_{\alpha,m}^*IC(X(\alpha),L_{\chi})$ is a simple perverse sheaf. 

Moreover, for the same reason, $\mathbf{H}^i(j_{\alpha}^*{\pi_{\underline{\nu}}}_!\overline{\bbQ_l})$ is generated by local systems $L_{\chi},$ $\chi\in X(S_m).$ 
\begin{corollary}
 The closure $\overline{X(\alpha,m)}$ of $X(\alpha,m)$ is the union of some $X(\alpha',m').$ 
\end{corollary}
\begin{proof}
 Considering $\mathcal{L}_{\alpha,\underline{\lambda}}$ in equation~\ref{alpha}, choosing $\underline{\lambda}=(m),$ we know $\mathbf{H}^i(j_{\alpha'}^*\mathcal{L}_{\alpha,\underline{\lambda}})$ is generated by $L_{\chi'}$ and $\operatorname{supp }\mathcal{L}_{\alpha,\underline{\lambda}}=\overline{X(\alpha,m)}.$ Thus if $j_{\alpha'}^*\mathcal{L}_{\alpha,\underline{\lambda}}\not=0,$ $X(\alpha')\subset \overline{X(\alpha,m)}.$ Moreover, $X(\alpha',m')\subset \overline{X(\alpha')}.$ Then $j_{\alpha'}^*\mathcal{L}_{\alpha,\underline{\lambda}}\not=0$ gives $X(\alpha',m')\subset \overline{X(\alpha,m)}.$ If $j_{\alpha'}^*\mathcal{L}_{\alpha,\underline{\lambda}}=0,$ by Theorem~\ref{th}, it follows that $j_{\alpha',m'}^*\mathcal{L}_{\underline{\lambda}}=0.$ For any geometric point $x\in \overline{X(\alpha,m)}\cap X(\alpha',m'),$ we obtain that $i_x^*\mathcal{L}_{\alpha,\underline{\lambda}}=0,$ contradicting  $\chi_n(\mathcal{L}_{\alpha,\underline{\lambda}})(x)\not=0,$ where $\chi_n$ is the trace map of $\operatorname{Fr_{q^n}}$ for $n\gg0.$ Thus $$\overline{X(\alpha)}=\cup_{(\alpha',m'),j_{\alpha'}^*\mathcal{L}_{\alpha,\underline{\lambda}}\not=0} X(\alpha',m').$$ 
\end{proof}
\section{The pointwise purity of $\mathcal{L}_0$}\label{4sec}
Let $r_X:X=X_0\times_{\operatorname{Spec}(\bbF_q)}\operatorname{Spec(\overline{\bbF_q})}\rightarrow X_0$ be the morphism induced by $\operatorname{Spec}(\overline{\bbF_q})\rightarrow \operatorname{Spec}(\bbF_q),$ and set $\operatorname{egf}=r_{X}^*:\cD^b_m(X_0,\overline{\bbQ_l})\rightarrow \cD^b_c(X,\overline{\bbQ_l}).$ By \cite[Lemma 5.3.8]{Pramod-2021}, $\operatorname{egf}(\mathcal{F})$ is equipped with a canonical Weil structure.

We will consider $$\mathcal{L}=\overline{\bbQ_l}|_{\bfE_{\bfV_{s_1}}}\diamond \cdots \diamond \overline{\bbQ_l}|_{\bfE_{\bfV_{s_r}}}
\diamond \overline{\bbQ_l}|_{\bfE_{\bfV_{\lambda_1}}}\diamond \overline{\bbQ_l}|_{\bfE_{\bfV_{\lambda_2}}}\cdots \diamond \overline{\bbQ_l}|_{\bfE_{\bfV_{\lambda_k}}}
\diamond \overline{\bbQ_l}|_{\bfE_{\bfV'_{s'_1}}}\diamond \cdots \diamond \overline{\bbQ_l}|_{\bfE_{\bfV'_{s'_t}}}$$ from equation~\ref{alpha} satisfying $\lambda_i=1$ which is a direct summand up to shift and Tate twist of the flag sheaf complex $L_{\underline{\nu}},$ where $\underline{\nu}=(\nu_1,\cdots,\nu_k),$ $\nu_i=i_1$ or $i_2.$ $IC(X(\alpha),L_{\chi})$ is a direct summand of $\mathcal{L}.$ Since the flag sheaf complex can be defined over $\bbF_q$, let $\mathcal{L}_0$ denote the corresponding element, $\operatorname{egf}(\mathcal{L}_0)=\mathcal{L}.$ We will prove that $\mathcal{L}_0$ is pointwise pure in the next proposition.
Recall $j_{\alpha'}:X(\alpha')\rightarrow \bfE_{\bfV}.$

\begin{proposition}\label{propp1}
 $\mathcal{L}_0$ is pointwise pure of weight $0,$ and $j_{\alpha,m}^*\mathcal{L}_0$ is pure.
\end{proposition}
This proposition will be proved in Section \ref{sec2}.

 Applying $\operatorname{egf},$ we obtain a Weil structure on ${\operatorname{Fr}^r}^*\mathcal{L}\rightarrow \mathcal{L},$ and, for an $\bbF_q$-point $x,$ $\mathbf{H}^i(\mathcal{L})_x\rightarrow \mathbf{H}^i(\mathcal{L})_x$ has eigenvalues of absolute value $q^{\frac{i}{2}}.$ By the purity of $j_{\alpha,m}^*\mathcal{L}_0$ and Theorem~\ref{th}, we already have that $$j_{\alpha',m'}^*\mathcal{L}=\oplus_{l,\chi'}j_{\alpha',m'}^*(IC(X(\alpha'),L_{\chi'}))^{\oplus s_{l,\chi'}}[l].$$
 Since $IC(X(\alpha),L_{\chi})$ is a direct summand of $\mathcal{L},$ we obtain $$j_{\alpha',m'}^*IC(X(\alpha),L_{\chi})=\oplus_{l'',\chi''}j_{\alpha',m'}^*(IC(X(\alpha'),L_{\chi'}))^{\oplus s_{l'',\chi''}}[l''].$$

 By \cite[Lemma 5.3.6]{Pramod-2021} and the proof in \cite[Lemma 5.3.8]{Pramod-2021}, we know that $\operatorname{Fr}^*$ is an equivalence of categories and preserves perverse t-structure. Furthermore, it preserves simple perverse sheaves. By $\operatorname{Fr}^{-1}(\overline{X(\alpha)}-X(\alpha))\subset \overline{X(\alpha)}-X(\alpha),$ and $\mathcal{L}=\oplus IC(X(\alpha),L_{\chi})^{\oplus K_{(1,\cdots,1),\chi}}\oplus L',$ with $\operatorname{supp}(L')\subset \overline{X(\alpha)}-X(\alpha)$ and a Kostka number $K_{(1,\cdots,1),\chi},$ the Weil structure on $\mathcal{L}$ can be restricted to $\oplus IC(X(\alpha),L_{\chi})^{\oplus K_{(1,\cdots,1),\chi}}.$ Applying ${j_{\alpha}^0}^*$ in Lemma~\ref{lem24} and the covering map, we see that this Weil structure is written as $\oplus(IC(X(\alpha),L_{\chi}),f_{\alpha,\chi})^{\oplus K_{(1,\cdots,1),\chi}},$ and is stable under Verdier duality. Thus it induces a Weil structure $f_{\alpha,\chi}$ on its direct summand $IC(X(\alpha),L_{\chi}).$ This is precisely the Weil structure chosen by Lusztig in \cite[Theorem 5.2]{lusztig1998canonical}. 
\begin{definition}\cite{lusztig1998canonical}
 The set of $\chi(IC(X(\alpha),L_{\chi}),f_{\alpha,\chi})$ is the canonical basis of the negative part of the quantized enveloping algebra of the Kronecker quiver. 
\end{definition}
In the next section, we give another proof different from Lusztig's, that $\chi(IC(X(\alpha),L_{\chi}),f_{\alpha,\chi})$ forms a basis.

It remains to understand how $j_{\alpha',m'}^*IC(X(\alpha),L_{\chi})=\oplus_{l'',\chi''}j_{\alpha',m'}^*(IC(X(\alpha'),L_{\chi'}))^{\oplus s_{l'',\chi''}}[l'']$ induces the Weil structure $f_{\alpha,\chi}$ on $\oplus_{l'',\chi''}j_{\alpha',m'}^*(IC(X(\alpha'),L_{\chi'}))^{\oplus s_{l'',\chi''}}[l''].$ 
\begin{remark}
 Since $\mathbf{H}^i(\mathcal{L})_x\rightarrow \mathbf{H}^i(\mathcal{L})_x$ has eigenvalues of absolute value $q^{\frac{i}{2}}$ for $\bbF_q$-point,
 
 \begin{equation}\label{bushi1}
[j_{\alpha',m'}^*(IC(X(\alpha),\chi),f_{\alpha,\chi})]=\sum_{\chi',k,i=1,\cdots,s_{k,\chi'}} [j_{\alpha',m'}^*(IC(X(\alpha'),L_{\chi'}),c_{\alpha',\chi',i}f_{\alpha',\chi'})[k](\frac{k}{2})],\end{equation} where $|c_{\alpha',\chi',i}|=1.$
\end{remark}
We consider the Hall algebra $\mathcal{H}_q$ of the Kronecker quiver, and recall that the negative part of the quantized enveloping algebra is isomorphic to the composition subalgebra of $\mathcal{H}_q,$ with $v=-\sqrt{q}$.

 $\mathcal{H}_q\cong \oplus_{\bfV}\cH_{\bfG_{\bfV}^{\operatorname{Fr}^r}}(\bfE_{\bfV}^{\operatorname{Fr}^r})$ by isomorphism class $[M]$ maps to $(-\sqrt{q})^{-\sum_{i\in I}\nu_i^2}1_{M},$ where $1_{M}\in \cH_{\bfG_{\bfV}^{\operatorname{Fr}^r}}(\bfE_{\bfV}^{\operatorname{Fr}^r})$ is a constant function on the orbit of $M.$

In $\cH_{\bfG_{\bfV}^{\operatorname{Fr}^r}}(\bfE_{\bfV}^{\operatorname{Fr}^r})$, we define the following functions.

For $\nu=(n,n)$, define $1_{(n,n)}$ by
\[
1_{(n,n)}(x)=
\begin{cases}
1,& M_x\text{ is regular},\\
0,& \text{otherwise},
\end{cases}
\qquad \text{where }\dim \bfV=\nu.
\]
For $1_{P_{s}^{a_{s}^P}}$, define
\[
1_{P_{s}^{a_{s}^P}}(x)=
\begin{cases}
1,& x\in \cO_{P_{s}^{a_{s}^P}}^{\operatorname{Fr}^r},\\
0,& \text{otherwise},
\end{cases}
\]
and define $1_{I_{s}^{a_s^I}}$ in the same way.

For $\alpha\in R$ and $\underline{\lambda}=(\lambda_1,\ldots,\lambda_k)$, $\lambda_1\geq \lambda_2\geq\cdots \geq \lambda_k$, set
\begin{equation}\label{eq11}
f^r_{\alpha,\underline{\lambda}}
=
1_{P_{s_1}^{a_{s_1}^P}}\cdots 1_{P_{s_r}^{a_{s_r}^P}}
\prod_{j=1}^k 1_{(\lambda_j,\lambda_j)}
1_{I_{s'_1}^{a_{s'_1}^I}}\cdots 1_{I_{s'_t}^{a_{s'_t}^I}},
\end{equation}
where
\[
\sum_{j=1}^r a_{s_j}^P P_{s_j}+\sum_{j=1}^t a_{s'_j}^I I_{s'_j}=\alpha,
\qquad
\text{and for }a<b,\ s_a\le s_b,\ s'_a\ge s'_b.
\]
Denote by $1_{\bfE_{\bfV}^{\operatorname{Fr}^r}}$ the constant function on $\bfE_{\bfV}^{\operatorname{Fr}^r}$. Let $\bfV_{s}$ be such that $P_{s}^{\oplus a_{s}^P}\in \bfE_{\bfV_{s}}$, let $\bfV'_{s}$ be such that $I_{s}^{\oplus a_{s}^I}\in \bfE_{\bfV'_{s}}$, and let $\bfV_{\lambda_i}$ have dimension $(\lambda_i,\lambda_i)$.

Define
\[
r^{r}_{\alpha,\underline{\lambda}}
=
1_{\bfE_{\bfV_{s_1}}}\cdots 1_{\bfE_{\bfV_{s_r}}}
\prod_{j=1}^k 1_{\bfE_{\bfV_{\lambda_j}}}
1_{\bfE_{\bfV'_{s'_1}}}\cdots 1_{\bfE_{\bfV'_{s'_t}}},
\]
where
\[
\sum_{j=1}^r a_{s_j}^P P_{s_j}+\sum_{j=1}^t a_{s'_j}^I I_{s'_j}=\alpha,
\qquad
\text{and for }a<b,\ s_a\le s_b,\ s'_a\ge s'_b.
\]
\begin{lemma}\cite[Theorem 5.1, Theorem 6.3]{+2000+97+116}\cite[Proposition 7.2]{linxiao}\label{cons}
$f^r_{\alpha,\underline{\lambda}}$ forms a basis of the negative part of the quantized enveloping algebra. For field $\bbF_q$, if $\underline{\nu}=(i_1,\ldots,i_n)$ with $i_j\in I$, then
\[
1_{\underline{\nu}}=\prod_{j=1}^n 1_{S_{i_j}},
\qquad
1_{\underline{\nu}}=\sum_{(\alpha,\underline{\lambda})} g_{\underline{\nu},(\alpha'',\underline{\lambda''})}(q)\, f^r_{(\alpha,\underline{\lambda})}.
\]
Moreover, $g_{\underline{\nu},(\alpha'',\underline{\lambda''})}(v)\in \bbZ[v].$
\end{lemma}
We now decompose the intersection cohomology into $j_{\alpha,m}^*IC(X(\alpha),L_{\chi})$, with Weil structures given by $j_{\alpha,m}^*f_{\alpha,\chi}$. For $\alpha\in R$, $m\in X(S_m)$, and $m\in \bbN$, we define the elements
\[
\chi({j_{\alpha,m}}_!j_{\alpha,m}^*IC(X(\alpha),L_{\chi}),{j_{\alpha,m}}_!j_{\alpha,m}^*f_{\alpha,\chi})
\]
as the PBW basis of the Kronecker quiver.

\begin{proposition}\label{thmm}
In the modified Grothendieck group, $c_{k,\chi',i}=1$ in equation~\ref{bushi1}.
\end{proposition}
\begin{proof}
We study the flag sheaf complex $\mathcal{L}$ in Proposition~\ref{propp}, with the Weil structure induced by $\mathcal{L}_0.$ 

By Lemma~\ref{cons} and equation~\ref{eq11}, after applying the trace map there exists a collection of polynomials $h_{\alpha',\underline{\lambda'}}(v)$ such that
\[
\chi(j_{\alpha',m'}^*\mathcal{L})
=
j_{\alpha',m'}^*\chi(\mathcal{L})
=
j_{\alpha',m'}^*\!\left(\sum_{\alpha'',\underline{\lambda''}} h_{\alpha'',\underline{\lambda''}}(q)\, f^r_{\alpha'',\underline{\lambda''}}\right)
=
\sum_{\alpha'',\underline{\lambda''}} h_{\alpha'',\underline{\lambda''}}(q)\, j_{\alpha',m'}^* f^r_{\alpha'',\underline{\lambda''}}.
\]

If $\alpha''\not=\alpha',$ $j_{\alpha',m'}^*f^r_{\alpha'',\underline{\lambda''}}=0.$ Since the elements $f^r_{\alpha'',\underline{\lambda''}}$ form a basis, fix $(\alpha,m),$ $\{f^r_{\alpha,\underline{\lambda}}|\underline{\lambda}=(\lambda_1,\cdots,\lambda_k),k\in \bbN_{\geq 1}, \sum_{i=1}^k \lambda_i=m, \lambda_j\geq\lambda_{j+1},j=1,\cdots,k-1\}$ are $\bbC$-linearly independent.

There also exist Kostka numbers $K_{\underline{\lambda'},\chi'}\in \bbN$ such that, for $\underline{\lambda'}=(\lambda'_1,\ldots,\lambda'_k)$ with $\sum \lambda'_j=m,$ $\lambda_i\geq \lambda_{i+1},\ i=1\cdots k-1$ and $\chi\in S_m$, we have
\[
j_{\alpha',m'}^*f^r_{\alpha',\underline{\lambda'}}
=
\chi({p_{\alpha,\underline{\lambda}}}_!\overline{\bbQ_l}[\dim X(\alpha)](\frac{\dim X(\alpha)}{2}))
=
\sum_{\chi} K_{\underline{\lambda'},\chi'}\,
\chi\!\left(j_{\alpha',m'}^*(IC(X(\alpha'),L_{\chi'}),f_{\alpha',\chi'})\right),
\]
which follows from Remark~\ref{rem25}. 
Since $j_{\alpha',m'}^*f^r_{\alpha',\underline{\lambda'}}$ are $\bbC$-linearly independent, by comparing dimensions, $\chi\!\left(j_{\alpha',m'}^*(IC(X(\alpha'),L_{\chi'}),f_{\alpha',\chi'})\right)$ are $\bbC$-linearly independent.

Thus $\chi(j_{\alpha',m'}^*\mathcal{L})\in \bbZ[\sqrt{q},\sqrt{q}^{-1}](\chi(j_{\alpha',m'}^*(IC(X(\alpha'),L_{\chi'}),f_{\alpha',\chi'})),\chi'\in X(S_{m'})),$ and is independent of the choice of $q$. 
 Choosing $\bbF_{q^n}$-points for $n>0$ yields the desired conclusion.
\end{proof}
\begin{remark}
 We have that ${j_{\alpha',m'}}_!j_{\alpha',m'}^*f^r_{\alpha',\underline{\lambda'}}=f^r_{\alpha',\underline{\lambda'}}.$ It follows that the set of $$\chi\!\left({j_{\alpha',m'}}_!j_{\alpha',m'}^*(IC(X(\alpha'),L_{\chi'}),f_{\alpha',\chi'})\right)$$ consists of a basis, which is the PBW basis in the function-theoretic setting. This is because the elements $f^r_{\alpha,\underline{\lambda}}$ form a basis and $K_{\underline{\lambda},\chi}$ is a Kostka number. 
\end{remark}
\section{The geometric realization of coefficients between PBW basis and canonical basis}
 In this section, we show that the transition coefficients between the PBW basis and the canonical basis have the following geometric interpretation: 
\[\begin{aligned}
[IC(X(\alpha),L_{\chi}),f_{\alpha,\chi}]
=\sum_{j\in \bbZ,(\alpha',m')\leq(\alpha,m),\chi'\in X(S_{m'})}
&\dim\Hom\!\left(\mathbf{H}^{j-\dim X(\alpha')}\!\bigl({j_{\alpha'}}^*IC(X(\alpha),L_{\chi})\bigr),\ 
L_{\chi'}\right)(-\sqrt{q})^{j}\\&
\bigl[{j_{\alpha',m'}}_!\, j_{\alpha',m'}^*\bigl(IC(X(\alpha'),L_{\chi'}),f_{\alpha',\chi'}\bigr)\bigr].\end{aligned}
\] This gives another proof, different from Lusztig's, that elements $\chi(IC(X(\alpha),L_{\chi}),f_{\alpha,\chi})$ form a basis.

We have already proved that $[j_{\alpha',m'}^*(IC(X(\alpha),L_\chi),f_{\alpha,\chi})]=\sum_{\chi',k}[j_{\alpha',m'}^*(IC(X(\alpha'),L_{\chi'}),f_{\alpha',\chi'}))[k](\frac{k}{2})].$ We now prove that the set of $[{j_{\alpha',m'}}_!j_{\alpha',m'}^*(IC(X(\alpha'),L_{\chi'}),f_{\alpha',\chi'})]$ is a $\bbZ[\sqrt{q},\sqrt{q}^{-1}]$-basis of $f_{\bfV}$.

By Proposition~\ref{propp1} and Theorem~\ref{th}, $j_{\alpha',m'}^*(IC(X(\alpha),L_{\chi}))$ has the form $$\oplus_{l,\chi'}j_{\alpha',m'}^*(IC(X(\alpha'),L_{\chi'}))^{\oplus s_{l,\chi'}}[l](\frac{l}{2}).$$

 If we denote
\begin{align*}
p^j_{(\alpha,\chi),(\alpha',\chi')}
:=&
\dim\Hom\!\left({}^p\mathbf{H}^j\!\bigl({j_{\alpha'}}^*IC(X(\alpha),L_{\chi})\bigr),\ 
L_{\chi'}[\dim X(\alpha')]\right),\\
&\dim\Hom\!\left(\mathbf{H}^{j-\dim X(\alpha')}\!\bigl({j_{\alpha'}}^*IC(X(\alpha),L_{\chi})\bigr),\ 
L_{\chi'}\right).
\end{align*}
and define the polynomial
\[
p_{(\alpha,\chi),(\alpha',\chi')}(v)
=
\sum_{j} p^j_{(\alpha,\chi),(\alpha',\chi')}\, v^j,
\]
then we have the following.

\begin{theorem}\label{main}
In $\mk^r$,
\[
[IC(X(\alpha),L_{\chi}),f_{\alpha,\chi}]
=
\sum_{(\alpha',m')\leq(\alpha,m),\chi'\in X(S_{m'})}
p_{(\alpha,\chi),(\alpha',\chi')}\!\left(-\sqrt{q}\right)\,
\bigl[{j_{\alpha',m'}}_!\, j_{\alpha',m'}^*\bigl(IC(X(\alpha'),L_{\chi'}),f_{\alpha',\chi'}\bigr)\bigr],
\]
for any $q=p^r$. Applying the trace map $\chi$, \[
\chi(IC(X(\alpha),L_{\chi}),f_{\alpha,\chi})
=
\sum_{(\alpha',m')\leq(\alpha,m),\chi'\in X(S_{m'})}
p_{(\alpha,\chi),(\alpha',\chi')}\!\left(-\sqrt{q}\right)\,
\chi({j_{\alpha',m'}}_!\, j_{\alpha',m'}^*\bigl(IC(X(\alpha'),L_{\chi'}),f_{\alpha',\chi'}\bigr)).
\] 
\end{theorem}

\begin{proof}
For a closed embedding $Z\xrightarrow{i} X$ with complement $U\xrightarrow{j} X,$ there is $$j_!j^*\mathcal{F}\rightarrow\mathcal{F}\rightarrow i_!i^*\mathcal{F}\xrightarrow{+1}.$$ If there is a closed embedding $Z'\xrightarrow{i'} Z$ with complement $U'\xrightarrow{j'} Z,$ $$i_!{j'}_!{j'}^*i^*\mathcal{F}\rightarrow i_!i^*\mathcal{F}\rightarrow i_!{i'}_!{i'}^*i^{*}\mathcal{F}\xrightarrow{+1}.$$ Therefore, for any stratification $\sqcup U_i=X, $ if we denote the embedding $U_i\xrightarrow{j_i} X,$ then in the modified Grothendieck group, we have $$[(\mathcal{F},f)]=\sum [(j_i)_!j_i^*(\mathcal{F},f)].$$ 
Since $\{X(\alpha',m')\mid \alpha'\in R,\ m'\in \bbN\}$ is a stratification of $\bfE_{\bfV}$ and $\overline{X(\alpha,m)}\subset \bigsqcup_{(\alpha',m')\leq (\alpha,m)}X(\alpha',m') $, we have in $\mk^r$ that
\[
[IC(X(\alpha),L_{\chi}),f_{\alpha,\chi}]
=
\sum_{(\alpha',m')\leq(\alpha,m)}
\bigl[{j_{\alpha',m'}}_!\, j_{\alpha',m'}^*\bigl(IC(X(\alpha),L_{\chi}),f_{\alpha,\chi}\bigr)\bigr].
\]

By Proposition~\ref{propp1} and~\ref{thmm}, we also have
\[
\bigl[j_{\alpha',m'}^*\bigl(IC(X(\alpha),L_{\chi}),f_{\alpha,\chi}\bigr)\bigr]
=
\sum_{\chi'\in X(S_{m'})}
p_{(\alpha,\chi),(\alpha',\chi')}\!\left(-\sqrt{q}\right)\,
\bigl[j_{\alpha',m'}^*\bigl(IC(X(\alpha'),L_{\chi'}),f_{\alpha',\chi'}\bigr)\bigr].
\]
Combining these equalities gives the desired formula.
\end{proof}
It follows that $\chi(IC(X(\alpha),L_{\chi}),f_{\alpha,\chi})$ is a basis of the negative part of the quantized enveloping algebra, since elements $\chi({j_{\alpha,m}}_!j_{\alpha,m}^*(IC(X(\alpha),L_{\chi}),f_{\alpha,\chi}))$ form a basis, and by Verdier duality, it is fixed under the bar involution, which gives the following theorem of Lusztig.
\begin{theorem}\cite[Theorem 6.16]{Lusztig1992}\cite[Theorem 6.2]{lusztig1998canonical}
The algebra $\bfU^-$ obtained from the Kronecker quiver is isomorphic to $\bigoplus_{\bfV} \chi(f_{\bfV})$, with the multiplication induced from $\mk^r$. Moreover, the set of $\chi(IC(X(\alpha),L_{\chi}),f_{\alpha,\chi})$ gives a basis which is fixed under the bar involution. 
\end{theorem}
We now show that for any odd number $l+\dim(X(\alpha))-\dim(X(\alpha')),$ ${}^p\mathbf{H}^l(j_{\alpha',m'}^*IC(X(\alpha),L_{\chi}))$ is zero.
\begin{proposition}\label{prop13}
 If $l+\dim(X(\alpha))-\dim(X(\alpha'))$ is odd, then ${}^p\mathbf{H}^l(j_{\alpha',m'}^*IC(X(\alpha),L_{\chi}))=0.$ And ${}^p\mathbf{H}^l(j_{\alpha',m'}^*IC(X(\alpha),L_{\chi}))=0,$ for $l\geq0$ if $\alpha'\not=\alpha.$
\end{proposition}
\begin{proof}
 As in \cite{lusztig1990canonical}, we consider three bases. We denote the PBW basis $\chi({j_{\alpha,m}}_!\, j_{\alpha,m}^*\bigl(IC(X(\alpha),L_{\chi}),f_{\alpha,\chi}\bigr))=Q_{\alpha,\chi},$ canonical basis $\chi(IC(X(\alpha),L_{\chi}),f_{\alpha,\chi})=B_{\alpha,\chi}$ and polynomial basis $P_{\alpha,\chi}=\sum_{\underline{\lambda}}c_{\chi,\underline{\lambda}}r_{\alpha,\underline{\lambda}}^r,$ where $(c_{\chi,\underline{\lambda}})_{\chi,\underline{\lambda}}$ is inverse matrix of $(K_{\underline{\lambda}},\chi)_{\underline{\lambda},\chi},$ thus fixed under bar involution.

 By Proposition 4.1 in \cite{MCGERTY2005411}, we can find this kind of polynomial basis with $c_{\chi,\underline{\lambda}}\in \bbQ.$

 By Theorem~\ref{main}, we have the following equations, \begin{align}
 B_{\alpha,\chi}&=Q_{\alpha,\chi}+\sum_{(\alpha',m')< (\alpha,m),\chi'\in X(S_{m'})}p_{(\alpha,\chi),(\alpha',\chi')}(-\sqrt{q})Q_{\alpha',\chi'},\\
 P_{\alpha,\chi}&=Q_{\alpha,\chi}+\sum_{(\alpha',m')< (\alpha,m),\chi'\in X(S_{m'})}u_{(\alpha,\chi),(\alpha',\chi')}(-\sqrt{q})Q_{\alpha',\chi'},\\
 \overline{Q_{\alpha,\chi}}&=Q_{\alpha,\chi}+\sum_{(\alpha',m')< (\alpha,m),\chi'\in X(S_{m'})}v_{(\alpha,\chi),(\alpha',\chi')}(-\sqrt{q})Q_{\alpha',\chi'},
 \end{align}
 where $p_{(\alpha,\chi),(\alpha',\chi')}(v)\in \bbZ[v,v^{-1}],$ $u_{(\alpha,\chi),(\alpha',\chi')}\in \bbQ[v,v^{-1}]$ and $v_{(\alpha,\chi),(\alpha',\chi')}\in \bbQ[v,v^{-1}].$ If $\alpha=\alpha',$ $p_{(\alpha,\chi),(\alpha',\chi')}(v)=u_{(\alpha,\chi),(\alpha',\chi')}(v)=v_{(\alpha,\chi),(\alpha',\chi')}(v)=\delta_{\chi,\chi'}.$
 The third equation follows from the second one and that polynomial basis is fixed under the bar involution.

 In the Hall algebra, we calculate that $u'_{(\alpha,\chi),(\alpha',\chi')}(v):=v^{\dim(X(\alpha))-\dim(X(\alpha'))}u_{(\alpha,\chi),(\alpha',\chi')}(v)\in \bbQ[v^2,v^{-2}].$

 Now applying bar involution on $B_{\alpha,\chi}$ and $P_{\alpha,\chi},$ we have the following equation,
 \begin{align}
 p_{(\alpha,\chi),(\alpha',\chi')}(v)&=\sum_{(\alpha',m')\leq(\alpha'',m'')\leq(\alpha,m),\chi''\in X(S_{m''})}\overline{p_{(\alpha,\chi),(\alpha'',\chi'')}(v)}v_{(\alpha'',\chi''),(\alpha',\chi')(v)},\\
 u_{(\alpha,\chi),(\alpha',\chi')}(v)&=\sum_{(\alpha',m')\leq(\alpha'',m'')\leq(\alpha,m),\chi''\in X(S_{m''})}\overline{u_{(\alpha,\chi),(\alpha'',\chi'')}(v)}v_{(\alpha'',\chi''),(\alpha',\chi')(v)}.
 \end{align}
 If we denote $v'_{(\alpha'',\chi''),(\alpha',\chi')}(v)=v^{\dim (X(\alpha'))-\dim(X(\alpha''))}v_{(\alpha'',\chi''),(\alpha',\chi')}(v)$
 Then $$u'_{(\alpha,\chi),(\alpha',\chi')}(v)=v^{2\dim(X(\alpha))-2\dim(X(\alpha'))}\sum_{(\alpha',m')\leq(\alpha'',m'')\leq(\alpha,m),\chi''\in X(S_{m''})}\overline{u'_{(\alpha,\chi),(\alpha'',\chi'')}(v)}v'_{(\alpha'',\chi''),(\alpha',\chi')}(v).$$
 Thus $v'_{(\alpha'',\chi''),(\alpha',\chi')}(v)\in \bbQ[v^2,v^{-2}].$
 
 If we denote $p'_{(\alpha'',\chi''),(\alpha',\chi')}(v)=v^{\dim (X(\alpha''))-\dim(X(\alpha'))}p_{(\alpha'',\chi''),(\alpha',\chi')}(v),$
 then $$p'_{(\alpha,\chi),(\alpha',\chi')}(v)=v^{2\dim(X(\alpha))-2\dim(X(\alpha'))}\sum_{(\alpha',m')\leq(\alpha'',m'')\leq(\alpha,m),\chi''\in X(S_{m''})}\overline{p'_{(\alpha,\chi),(\alpha'',\chi'')}(v)}v'_{(\alpha'',\chi''),(\alpha',\chi')}(v).$$
 Since we already know that $p'_{(\alpha,\chi),(\alpha',\chi')}(v)\in \bbQ[v^2,v^{-2}]$ when $\alpha=\alpha',$ by descending induction, we have $p'_{(\alpha,\chi),(\alpha',\chi')}(v)-v^{2\dim(X(\alpha))-2\dim(X(\alpha'))}\overline{p'_{(\alpha,\chi),(\alpha',\chi')}(v)}\in \bbQ[v^2,v^{-2}].$

 Since ${j_{\alpha'}}^*(IC(X(\alpha),L_{\chi}))\in {}^p\cD^{\leq -1}(X(\alpha'))$ by \cite[Lemma 3.3.4]{Pramod-2021}\cite[Lemma 6.1]{laszlo2009} $p'_{(\alpha,\chi),(\alpha',\chi')}(v)\in v^{\dim(X(\alpha))-\dim(X(\alpha'))-1}\bbQ[v^{-1}].$ We obtain $$v^{2\dim(X(\alpha))-2\dim(X(\alpha'))}\overline{p'_{(\alpha,\chi),(\alpha',\chi')}(v)}\in v^{\dim(X(\alpha))-\dim(X(\alpha'))+1}\bbQ[v].$$ Then we prove that $$p'_{(\alpha,\chi),(\alpha',\chi')}(v)\in \bbQ[v^2,v^{-2}].$$
\end{proof}
\begin{corollary}
 $p_{(\alpha,\chi),(\alpha',\chi')}(v)$ has the following properties, \\
 $\bullet$ $p_{(\alpha,\chi),(\alpha,\chi)}(v)=1,$\\
 $\bullet$ $p_{(\alpha,\chi),(\alpha',\chi')}(v)\in v^{-1}\mathbb{N}[v^{-1}]$ for $(\alpha',m')<(\alpha,m)$ and $p_{(\alpha,\chi),(\alpha',\chi')}(v)=0$ for $(\alpha',m')>(\alpha,m),$\\
 $\bullet$ $v^{-\dim(X(\alpha))+\dim(X(\alpha'))}p_{(\alpha,\chi),(\alpha',\chi')}(v)\in \bbN[v^{-2}]$.
\end{corollary}
\begin{proof}
 This corollary follows from the definition of $p_{(\alpha,\chi),(\alpha',\chi')}(v)$ and Proposition~\ref{prop13}.
\end{proof}
\begin{remark}
 For Dynkin quivers, there are similar results in \cite{lusztig1990canonical} and \cite{cd088214-009e-3813-9a86-97f9a299a24c}, and for finite symmetrizable cases, there are similar results in \cite{lan2025structurecoefficientsquantumgroups}.
\end{remark}
By \cite{MCGERTY2005411} and Proposition~\ref{propp}, we know, by using the trace map, that our PBW basis is the same as the basis defined by Nakajima and Beck in \cite{2002Crystal}, and we give a geometric expression of the coefficients between the canonical basis and the PBW basis and establish further properties of these coefficients in this article. This paper gives another proof of $p_{(\alpha,\chi),(\alpha',\chi')}(v)\in v^{-1}\mathbb{N}[v^{-1}]$ from \cite{McNamara2017} in the Kronecker quiver case, and moreover shows $v^{-\dim(X(\alpha))+\dim(X(\alpha'))}p_{(\alpha,\chi),(\alpha',\chi')}(v)\in \bbN[v^{-2}]$. 

\section{The proof of Theorem~\ref{th}}\label{sec}
To obtain a more concrete description of $j_{\alpha,m}^*\mathcal{L}_{\underline{\nu}}$, we need the following lemmas.

\begin{lemma}
For $\bfV$ with $\dim \bfV=(n,n-1)$, the variety
\[
K=\left\{(x,f)\ \middle|\ 
\begin{aligned}
&x\in \bfE_{\bfV},\ f=(0=V_0\subset V_1\subset V_2=\bfV)\text{ is a filtration of }\bbN I\text{-graded spaces},\\
&x(V_j)\subset V_j,\ M_{x|_{V_2/V_1}}\cong S_1,\ M_{x|_{V_1}}\text{ is a direct sum of indecomposable regular modules}
\end{aligned}\right\}
\]
is isomorphic to the complement in
\[
\bigsqcup_{s'\ge 1}\left\{(x,f)\ \middle|\ 
\begin{aligned}
&x\in \bfE_{\bfV},\ f=(0=V_0\subset V_1\subset V_2=\bfV)\text{ is a filtration of }\bbN I\text{-graded spaces},\\
&M_{x|_{V_1}}\cong I_{s'},\ M_{x|_{V_2/V_1}}\text{ is a direct sum of indecomposable regular modules}
\end{aligned}\right\}\times \mathbb{P}_{n-1}
\]
of the closed subvariety
\[
\bigsqcup_{i',s'\ge 1}\left\{(x,f)\ \middle|\ 
\begin{aligned}
&x\in \bfE_{\bfV},\ f=(0=V_0\subset V_1\subset V_2=\bfV)\text{ is a filtration of }\bbN I\text{-graded spaces},\\
&x(V_j)\subset V_j,\ M_{x|_{V_1}}\cong I_{s'},\ M_{x|_{V_2/V_1}}\text{ is a direct sum of indecomposable regular modules}
\end{aligned}\right\}\times \bbP_{n-s'-1}.
\]
\end{lemma}

\begin{proof}
We have
\[
K=\left\{(x,f)\ \middle|\ 
\begin{aligned}
&x\in \bfE_{\bfV},\ f=(0=V_0\subset V_1\subset V_2=\bfV)\text{ is a filtration of }\bbN I\text{-graded spaces},\\
&x(V_j)\subset V_j,\ M_{x|_{V_2/V_1}}\cong S_1,\ M_{x|_{V_1}}\text{ is a direct sum of indecomposable regular modules}
\end{aligned}\right\},
\]
which is isomorphic to the complement in
\[
\left\{(x,f)\ \middle|\ 
\begin{aligned}
&x\in \bfE_{\bfV},\ f=(0=V_0\subset V_1\subset V_2=\bfV)\text{ is a filtration of }\bbN I\text{-graded spaces},\\
&x(V_j)\subset V_j,\ M_{x|_{V_2/V_1}}\cong S_1
\end{aligned}\right\}
\]
of the closed subvariety
\[
\left\{(x,f)\ \middle|\ 
\begin{aligned}
&x\in \bfE_{\bfV},\ f=(0=V_0\subset V_1\subset V_2=\bfV)\text{ is a filtration of }\bbN I\text{-graded spaces},\\
&x(V_j)\subset V_j,\ M_{x|_{V_2/V_1}}\cong S_1,\ M_{x|_{V_1}}\text{ is not a direct sum of indecomposable regular modules}
\end{aligned}\right\}.
\]
For any short exact sequence of representations of $Q$,
\[
0\rightarrow H_r\rightarrow M\rightarrow S_1\rightarrow 0,
\]
where $H_r$ is a direct sum of indecomposable regular modules, we have that $M$ has no preprojective direct summands.

Thus, $K$ is isomorphic to the complement in
\[
\left\{x\ \middle|\ x\in \bfE_{\bfV},\ M_x\text{ has no preprojective direct summands}\right\}\times \mathbb{P}_{n-1}
\]
of the closed subvariety
\[
K'=\left\{(x,f)\ \middle|\ 
\begin{aligned}
&x\in \bfE_{\bfV},\ M_x\text{ has no preprojective direct summands},\\
&f=(0=V_0\subset V_1\subset V_2=\bfV)\text{ is a filtration of }\bbN I\text{-graded spaces},\\
&x(V_j)\subset V_j,\ M_{x|_{V_2/V_1}}\cong S_1,\\
&M_{x|_{V_1}}\cong P_{s}\oplus H_r\oplus I_{s'},\ H_r\text{ is a direct sum of indecomposable regular modules}
\end{aligned}\right\}.
\]

Moreover, with the topological structure given by Lemma \ref{stra}, $K'$ is isomorphic to
\[
\bigsqcup_{s'}\left\{(x,f)\ \middle|\ 
\begin{aligned}
&x\in \bfE_{\bfV},\ M_x\text{ has no preprojective direct summands},\\&f=(0=V_0\subset V_1\subset V_2\subset V_3=\bfV)\text{ is a filtration of }\bbN I\text{-graded spaces},\\
&x(V_j)\subset V_j,\ M_{x|_{V_3/V_2}}\cong S_1,\\
& M_{x|_{V_1}}\cong I_{s'}
\end{aligned}\right\},
\]
which is, in turn, isomorphic to
\[
\bigsqcup_{s'}\left\{(x,f)\ \middle|\ 
\begin{aligned}
&x\in \bfE_{\bfV},\ M_x\text{ has no preprojective direct summands},\\&f=(0=V_0\subset V_1\subset V_2=\bfV)\text{ is a filtration of }\bbN I\text{-graded spaces},\\
&x(V_j)\subset V_j,\ M_{x|_{V_1}}\cong I_{s'}
\end{aligned}\right\}\times \bbP_{n-s'-1}.
\]
Computing the dimension of $V_2/V_1$, the variety above is isomorphic to
\[
\bigsqcup_{s'}\left\{(x,f)\ \middle|\ 
\begin{aligned}
&x\in \bfE_{\bfV},\ M_x\text{ has no preprojective direct summands},\\&f=(0=V_0\subset V_1\subset V_2=\bfV)\text{ is a filtration of }\bbN I\text{-graded spaces},\\
&x(V_j)\subset V_j,\ M_{x|_{V_1}}\cong I_{s'},\ M_{x|_{V_2/V_1}}\text{ is a direct sum of indecomposable regular modules}
\end{aligned}\right\}\times \bbP_{n-s'-1}.
\]

Therefore, $K$ is isomorphic to the complement in
\[
\bigsqcup_{s'}\left\{(x,f)\ \middle|\ 
\begin{aligned}
&x\in \bfE_{\bfV},\ f=(0=V_0\subset V_1\subset V_2=\bfV)\text{ is a filtration of }\bbN I\text{-graded spaces},\\
&M_{x|_{V_1}}\cong I_{s'},\ M_{x|_{V_2/V_1}}\text{ is a direct sum of indecomposable regular modules}
\end{aligned}\right\}\times \mathbb{P}_{n-1}
\]
of the closed subvariety
\[
\bigsqcup_{s'}\left\{(x,f)\ \middle|\ 
\begin{aligned}
&x\in \bfE_{\bfV},\ f=(0=V_0\subset V_1\subset V_2=\bfV)\text{ is a filtration of }\bbN I\text{-graded spaces},\\
&x(V_j)\subset V_j,\ M_{x|_{V_1}}\cong I_{s'},\ M_{x|_{V_2/V_1}}\text{ is a direct sum of indecomposable regular modules}
\end{aligned}\right\}\times \bbP_{n-s'-1}.
\]

Finally,
\[
\left\{(x,f)\ \middle|\ 
\begin{aligned}
&x\in \bfE_{\bfV},\ f=(0=V_0\subset V_1\subset V_2=\bfV)\text{ is a filtration of }\bbN I\text{-graded spaces},\\
&M_{x|_{V_1}}\cong I_{s'},\ M_{x|_{V_2/V_1}}\text{ is a direct sum of indecomposable regular modules}
\end{aligned}\right\}
\]
can be identified with 
\[
\left\{x\ \middle|\ 
\begin{aligned}
&x\in \bfE_{\bfV},\\
&M_x\cong I_{s'}\oplus H_r,\ \text{where } H_r\text{ is a direct sum of indecomposable regular modules}
\end{aligned}\right\}.
\]
\end{proof}
For $\bfV$ with dimension $(n,n),$ we denote \[W_{I_{s'},(n-s')}:=
\left\{x\ \middle|\ 
\begin{aligned}
&x\in \bfE_{\bfV},\\
&M_x\cong I_{s'}\oplus H_r,\ \text{where } H_r\text{ is a direct sum of indecomposable regular modules}
\end{aligned}\right\}.
\]

\begin{lemma}\label{2.4}
Let $\bfV$ be an $\bbN I$-graded space $\bfV$ with $\dim\bfV=(n-1,n)$, the variety
\[
K_s=\left\{(x,f)\ \middle|\ 
\begin{aligned}
&x\in \bfE_{\bfV},\ f=(0=V_0\subset V_1\subset V_2=\bfV)\text{ is a filtration of }\bbN I\text{-graded spaces},\\
&x(V_j)\subset V_j,\ M_{x|_{V_2/V_1}}\cong I_{s},\ M_{x|_{V_1}}\text{ is a direct sum of indecomposable regular modules}
\end{aligned}\right\}
\]
is isomorphic to $\bigsqcup_{t\ge s}\bigl(W'_{I_{t},(n-t)}-W''_{I_{t},(n-t)}\bigr)$. The right-hand side is a stratification, and $W''_{I_{t},(n-t)}$ is closed in $W'_{I_{t},(n-t)}$. Both $W''_{I_{t},(n-t)}$ and $W'_{I_{t},(n-t)}$ admit locally trivial fibrations over $W_{I_{t},(n-t)}$, and both fibers are projective spaces.
\end{lemma}

\begin{proof}
The variety $K_s$ is isomorphic to the complement in
\[
\left\{(x,f)\ \middle|\ 
\begin{aligned}
&x\in \bfE_{\bfV},\ f=(0=V_0\subset V_1\subset V_2=\bfV)\text{ is a filtration of }\bbN I\text{-graded spaces},\\
&x(V_j)\subset V_j,\ M_{x|_{V_2/V_1}}\cong I_{s}
\end{aligned}\right\}
\]
of the closed subvariety
\[
\left\{(x,f)\ \middle|\ 
\begin{aligned}
&x\in \bfE_{\bfV},\ f=(0=V_0\subset V_1\subset V_2=\bfV)\text{ is a filtration of }\bbN I\text{-graded spaces},\\
&x(V_j)\subset V_j,\ M_{x|_{V_2/V_1}}\cong I_{s},\ M_{x|_{V_1}}\text{ is not a direct sum of indecomposable regular modules}
\end{aligned}\right\}.
\]
By Lemma~\ref{le}, any extension of $I_{s}$ by a regular module must be of the form a direct sum of an indecomposable preinjective representation $I_{t}$, where $t\ge s$, and a regular representation.

Therefore, $K_s$ is isomorphic to the complement in
\[
\tilde{V}=\bigsqcup_{t\ge s}\left\{(x,f)\ \middle|\ 
\begin{aligned}
&M_x\cong I_{t}\oplus H_r,\ f=(0=V_0\subset V_1\subset V_2=\bfV)\text{ is a filtration of }\bbN I\text{-graded spaces},\\
&x(V_j)\subset V_j,\ M_{x|_{V_2/V_1}}\cong I_{s},\ H_r\text{ is a direct sum of indecomposable regular modules}
\end{aligned}\right\}
\]
of the closed subvariety
\[
V=\bigsqcup_{t\ge s}\left\{(x,f)\ \middle|\ 
\begin{aligned}
&M_x\cong I_{t}\oplus H_r,\ f=(0=V_0\subset V_1\subset V_2=\bfV)\text{ is a filtration of }\bbN I\text{-graded spaces},\\
&x(V_j)\subset V_j,\ M_{x|_{V_2/V_1}}\cong I_{s},\\
&M_{x|_{V_1}}\cong P_{u}\oplus H_r\oplus I_{t},\ 1\le u\le n-s-t+1,\\&H_r\text{ is a direct sum of indecomposable regular modules}
\end{aligned}\right\}.
\]
We know by Lemma \ref{stra}, fixed $t'\ge s,$  \[
\bigsqcup_{t'\ge t\ge s}\left\{(x,f)\ \middle|\ 
\begin{aligned}
&M_x\cong I_{t}\oplus H_r,\ f=(0=V_0\subset V_1\subset V_2=\bfV)\text{ is a filtration of }\bbN I\text{-graded spaces},\\
&x(V_j)\subset V_j,\ M_{x|_{V_2/V_1}}\cong I_{s},\\
&M_{x|_{V_1}}\cong P_{u}\oplus H_r\oplus I_{t},\ 1\le u\le n-s-t+1,\ H_r\text{ is a direct sum of indecomposable regular modules}
\end{aligned}\right\}
\]
is a closed subvariety in $V.$ And \[
\bigsqcup_{t'\ge t\ge s}\left\{(x,f)\ \middle|\ 
\begin{aligned}
&M_x\cong I_{t}\oplus H_r,\ f=(0=V_0\subset V_1\subset V_2=\bfV)\text{ is a filtration of }\bbN I\text{-graded spaces},\\
&x(V_j)\subset V_j,\ M_{x|_{V_2/V_1}}\cong I_{s},\ H_r\text{ is a direct sum of indecomposable regular modules}
\end{aligned}\right\}
\] is a closed subvariety in $\tilde{V}.$
Moreover,
\[
\left\{(x,f)\ \middle|\ 
\begin{aligned}
&M_x\cong I_{t}\oplus H_r,\ f=(0=V_0\subset V_1\subset V_2=\bfV)\text{ is a filtration of }\bbN I\text{-graded spaces},\\
&x(V_j)\subset V_j,\ M_{x|_{V_2/V_1}}\cong I_{s},\\
&M_{x|_{V_1}}\cong P_{u}\oplus H_r\oplus I_{t},\ 1\le u\le n-s-t+1,\ H_r\text{ is a direct sum of indecomposable regular modules}
\end{aligned}\right\}
\]
is isomorphic to
\[
\left\{(x,f)\ \middle|\ 
\begin{aligned}
&M_x\cong I_{t}\oplus H_r,\ f=(0=V_0\subset V_1\subset V_2\subset V_3=\bfV)\text{ is a filtration of }\bbN I\text{-graded spaces},\\
&x(V_j)\subset V_j,\ M_{x|_{V_3/V_2}}\cong I_{s},\ M_{x|_{V_1}}\cong I_{t},\\
& M_{x|_{V_2/V_1}}\cong P_{u}\oplus H_r,\ 1\le u\le n-s-t+1,\ H_r\text{ is a direct sum of indecomposable regular modules}
\end{aligned}\right\},
\]
which is also isomorphic to
\[
\left\{(x,f)\ \middle|\ 
\begin{aligned}
&M_x\cong I_{t}\oplus H_r,\ f=(0=V_0\subset V_1\subset V_2\subset V_3=\bfV)\text{ is a filtration of }\bbN I\text{-graded spaces},\\
&x(V_j)\subset V_j,\ M_{x|_{V_3/V_2}}\cong I_{s},\\
&M_{x|_{V_1}}\cong I_{t},\ M_{x|_{V_3/V_1}}\cong H_r,\ H_r\text{ is a direct sum of indecomposable regular modules}
\end{aligned}\right\}.
\]

Fix an $I$-graded subspace $\bfV'\subset \bfV$ with dimension $(t-1,t)$, and let $P\subset \bfG_{\bfV}$ be its stabilizer. Then the variety above is isomorphic to
\[
\bfG_V\times_{P}\left\{(x,f)\ \middle|\ 
\begin{aligned}
&x\in \bfE_{\bfV},\ f=(0=V_0\subset V_1=\bfV'\subset V_2\subset V_3=\bfV)\text{ is a filtration of }\bbN I\text{-graded spaces},\\
&x(V_j)\subset V_j,\ M_{x|_{V_3/V_2}}\cong I_{s},\ M_{x|_{V_1}}\cong I_{t},\ M_{x|_{V_3/V_1}}\cong H_r,\\
&H_r\text{ is a direct sum of indecomposable regular modules}
\end{aligned}\right\}.
\]
Now fix a representation $I=(x_I,\bfV_I)\cong I_{s}$.

Finally, since
\[
\left\{(x,f)\ \middle|\ 
\begin{aligned}
&M_x\cong I_{t}\oplus H_r,\ f=(0=V_0\subset V_1\subset V_2=\bfV)\text{ is a filtration of }\bbN I\text{-graded spaces},\\
&x(V_j)\subset V_j,\ M_{x|_{V_2/V_1}}\cong I_{s},\ H_r\text{ is a direct sum of indecomposable regular modules}
\end{aligned}\right\}
\]
is isomorphic to
\[
\left\{(x,g)\ \middle|\ 
\begin{aligned}
&M_x\cong I_{t}\oplus H_r,\ g\in \Hom_{kQ}(M_x,I)\text{ is surjective}\\
&H_r\text{ is a direct sum of indecomposable regular modules}
\end{aligned}\right\}/k^*,
\]
And for
\[
\bfG_V\times_{P}\left\{(x,f)\ \middle|\ 
\begin{aligned}
&x\in \bfE_{\bfV},\ f=(0=V_0\subset V_1=\bfV'\subset V_2\subset V_3=\bfV)\text{ is a filtration of }\bbN I\text{-graded spaces},\\
&x(V_j)\subset V_j,\ M_{x|_{V_3/V_2}}\cong I_{s},\ M_{x|_{V_1}}\cong I_{t},\ M_{x|_{V_3/V_1}}\cong H_r,\\
&H_r\text{ is a direct sum of indecomposable regular modules}
\end{aligned}\right\},
\]
we consider
\[\left\{(x,f)\ \middle|\ 
\begin{aligned}
&x\in \bfE_{\bfV},\ f=(0=V_0\subset V_1=\bfV'\subset V_2\subset V_3=\bfV)\text{ is a filtration of }\bbN I\text{-graded spaces},\\
&x(V_j)\subset V_j,\ M_{x|_{V_3/V_2}}\cong I_{s},\ M_{x|_{V_1}}\cong I_{t},\ M_{x|_{V_3/V_1}}\cong H_r,\\
&H_r\text{ is a direct sum of indecomposable regular modules}
\end{aligned}\right\}\] is isomorphic to \[
\left\{(x,g)\ \middle|\ 
\begin{aligned}
&x\in \bfE_{\bfV},\ M_x\cong H_r\oplus I_t,M_{x|_{\bfV'}}\cong I_t,\ g\in \Hom_{kQ}(H_r,I)\text{ is surjective},\\
&H_r\text{ is a direct sum of indecomposable regular modules}
\end{aligned}\right\}/k^*.
\]
Since the complement of \[
\left\{(x,g)\ \middle|\ 
\begin{aligned}
&x\in \bfE_{\bfV},\ M_x\cong H_r\oplus I_t,M_{x|_{\bfV'}}\cong I_t,\ g\in \Hom_{kQ}(H_r,I)\text{ is surjective}\\
&H_r\text{ is a direct sum of indecomposable regular modules}
\end{aligned}\right\}/k^*.
\] in \[
\left\{(x,g)\ \middle|\ 
\begin{aligned}
&x\in \bfE_{\bfV},\ M_x\cong I_{t}\oplus H_r,\ M_{x|_{\bfV'}}\cong I_t,\ g\in \Hom_{kQ}(M_x,I)\text{ is surjective}\\
&H_r\text{ is a direct sum of indecomposable regular modules}
\end{aligned}\right\}/k^*,
\] is the same as the complement of \[
\left\{(x,g)\ \middle|\ 
\begin{aligned}
&x\in \bfE_{\bfV},\ M_x\cong H_r\oplus I_t,M_{x|_{\bfV'}}\cong I_t,\ 0\not=g\in \Hom_{kQ}(H_r,I),\\
&H_r\text{ is a direct sum of indecomposable regular modules}
\end{aligned}\right\}/k^*
\] in \[
\left\{(x,g)\ \middle|\ 
\begin{aligned}
&x\in \bfE_{\bfV},\ M_x\cong I_{t}\oplus H_r,M_{x|_{\bfV'}}\cong I_t,\ 0\not=g\in \Hom_{kQ}(M_x,I),\\
&H_r\text{ is a direct sum of indecomposable regular modules}
\end{aligned}\right\}/k^*.
\]
Since
\[
\left\{(x,g)\ \middle|\ 
\begin{aligned}
&x\in \bfE_{\bfV},\ M_x\cong I_{t}\oplus H_r,M_{x|_{\bfV'}}\cong I_t, \ g\in \Hom_{kQ}(M_x,I),\\
&H_r\text{ is a direct sum of indecomposable regular modules}
\end{aligned}\right\}
\]
is the kernel of a morphism of vector bundles on $W_{I_{t},(n-t)}$ with constant rank for any $x\in W_{I_{t},(n-t)}$, it follows that it admits a locally trivial fibration over $W_{I_{t},(n-t)}$. Hence
\[
\left\{(x,g)\ \middle|\ 
\begin{aligned}
&x\in \bfE_{\bfV},\ M_x\cong I_{t}\oplus H_r,\ 0\not=g\in \Hom_{kQ}(M_x,I),\\
&H_r\text{ is a direct sum of indecomposable regular modules}
\end{aligned}\right\}/k^*
\]
also admits a locally trivial fibration over $W_{I_{t},(n-t)},$ whose fiber is a projective space. For \[
\left\{(x,g)\ \middle|\ 
\begin{aligned}
&x\in \bfE_{\bfV},\ M_x\cong H_r\oplus I_t,M_{x|_{\bfV'}}\cong I_t,\ 0\not=g\in \Hom_{kQ}(H_r,I),\\
&H_r\text{ is a direct sum of indecomposable regular modules}
\end{aligned}\right\}/k^*
\] we have the same result, thus applying $\bfG_{\bfV}\times_P$, we have that \[
\bfG_V\times_{P}\left\{(x,f)\ \middle|\ 
\begin{aligned}
&x\in \bfE_{\bfV},\ f=(0=V_0\subset V_1=\bfV'\subset V_2\subset V_3=\bfV)\text{ is a filtration of }\bbN I\text{-graded spaces},\\
&x(V_j)\subset V_j,\ M_{x|_{V_3/V_2}}\cong I_{s},\ M_{x|_{V_1}}\cong I_{t},\ M_{x|_{V_3/V_1}}\cong H_r,\\
&H_r\text{ is a direct sum of indecomposable regular modules}
\end{aligned}\right\}.
\] admits a locally trivial fibration over $W_{I_{t},(n-t)}.$ In fact, the fiber is a projective space.
This completes the proof.
\end{proof}
It is straightforward to see that $W'_{I_{t},(n-t)}$ and $W''_{I_{t},(n-t)}$ are both irreducible.

By induction, for $\underline{\lambda}=(\lambda_1,\ldots,\lambda_k)$ with $\sum_{j=1}^k \lambda_j=n$, $\lambda_i>0,$ we denote by $L_{\underline{\lambda}}$ the variety
\[
\left\{(x,f)\ \middle|\ 
\begin{aligned}
&x\in X(0,n),\ f=(0=V_0\subset V_1\subset V_2\subset \cdots \subset V_n=\bfV)\text{ is a filtration of }\bbN I\text{-graded spaces},\ x(V_i)\subset V_i,\\
&\dim V_i/V_{i-1}=(\lambda_i,\lambda_i),\ 
M_{x|_{V_i/V_{i-1}}}\text{ is a direct sum of indecomposable regular modules},\ i=1,\ldots,n
\end{aligned}\right\}\times \bbP_{\lambda_k-1}.
\]
We introduce an order by declaring $\underline{\lambda}\le \underline{\lambda}'$ if and only if $\underline{\lambda}$ is a refinement of $\underline{\lambda}'$. We then define a new variety inductively: $\widetilde{L_{\underline{\lambda}}}$ is the complement in $L_{\underline{\lambda}}$ of the closed subvariety
\[
\bigsqcup_{{\lambda_i}=\lambda'_{i},\ i\le k-1,\ \lambda_k=\lambda'_{k-1}+\lambda'_{k}} \widetilde{L_{\underline{\lambda}'}},
\]
where the closed embedding is defined by~\eqref{1} in the proof of the following lemma. Moreover, both $L_{\underline{\lambda}}$ and $\widetilde{L_{\underline{\lambda}}}$ are irreducible.

\begin{lemma}
For $\nu=(n,n)\in \bbN I$ with $\dim \bfV=\nu$, $s=n,$ the variety
\[
L=\left\{(x,f)\ \middle|\ 
\begin{aligned}
&x\in \bfE_{\bfV},\ f=(0=V_0\subset V_1\subset V_2=\bfV)\text{ is a filtration of }\bbN I\text{-graded spaces},\\
&x(V_i)\subset V_i,\ M_{x|_{V_1}}\cong P_{s},\ M_{x|_{V_2/V_1}}\cong S_1
\end{aligned}\right\}
\]
is isomorphic to $K\sqcup L_{(n)}$, where $K$ is closed in this variety and admits a locally trivial bundle over the orbit of $P_{s}\oplus S_1$.
\end{lemma}

\begin{proof}
If $0\rightarrow P_{s}\rightarrow M\rightarrow I_{s'}\rightarrow 0$ is a short exact sequence of representations of $Q$, then either $M\cong P_{s}\oplus I_{s'}$, or $M$ is a direct sum of regular indecomposable representations.

It is straightforward to see that
\[
\left\{(x,f)\ \middle|\ 
\begin{aligned}
&M_x\cong P_{s}\oplus S_1,\ f=(0=V_0\subset V_1\subset V_2=\bfV)\text{ is a filtration of }\bbN I\text{-graded spaces},\\
&x(V_i)\subset V_i,\ M_{x|_{V_1}}\cong P_{s},\ M_{x|_{V_2/V_1}}\cong S_1
\end{aligned}\right\}
\]
admits a locally trivial fibration over $\{x\in \bfE_{\bfV}\mid M_x\cong P_{s}\oplus S_1\}$, which is a closed subvariety of $L$. Thus it suffices to consider the case where $M$ is a direct sum of regular indecomposable representations.

We now give a more explicit description. 
\[
L=\left\{(x,f)\ \middle|\ 
\begin{aligned}
&x\in \bfE_{\bfV},\ f=(0=V_0\subset V_1\subset V_2=\bfV)\text{ is a filtration of }\bbN I\text{-graded spaces},\\
&x(V_i)\subset V_i,\ M_{x|_{V_1}}\cong P_{s},\ M_{x|_{V_2/V_1}}\cong S_1
\end{aligned}\right\},
\]
which is isomorphic to the complement in
\[
\left\{(x,f)\ \middle|\ 
\begin{aligned}
&x\in \bfE_{\bfV},\ f=(0=V_0\subset V_1\subset V_2=\bfV)\text{ is a filtration of }\bbN I\text{-graded spaces},\\
&x(V_i)\subset V_i,\ M_{x|_{V_2/V_1}}\cong S_1
\end{aligned}\right\}
\]
of the closed subvariety
\[
\left\{(x,f)\ \middle|\ 
\begin{aligned}
&x\in \bfE_{\bfV},\ f=(0=V_0\subset V_1\subset V_2=\bfV)\text{ is a filtration of }\bbN I\text{-graded spaces},\\
&x(V_i)\subset V_i,\ M_{x|_{V_1}}\not\cong P_{s},\ M_{x|_{V_2/V_1}}\cong S_1
\end{aligned}\right\}.
\]

After restricting to $x\in X(0,n)$, this is the complement of
\[
\left\{(x,f)\ \middle|\ 
\begin{aligned}
&x\in X(0,n),\ f=(0=V_0\subset V_1\subset V_2=\bfV)\text{ is a filtration of }\bbN I\text{-graded spaces},\\
&x(V_i)\subset V_i,\ M_{x|_{V_2/V_1}}\cong S_1
\end{aligned}\right\}
\]
of the closed subvariety
\[
\left\{(x,f)\ \middle|\ 
\begin{aligned}
&x\in X(0,n),\ f=(0=V_0\subset V_1\subset V_2=\bfV)\text{ is a filtration of }\bbN I\text{-graded spaces},\\
&x(V_i)\subset V_i,\ M_{x|_{V_1}}\not\cong P_{s},\ M_{x|_{V_2/V_1}}\cong S_1
\end{aligned}\right\}.
\]
The variety above is isomorphic to the complement in $\{(x,f)\mid x\in X(0,n)\}\times \bbP_{n-1}$ of the closed subvariety
\[
L'=\left\{(x,f)\ \middle|\ 
\begin{aligned}
&x\in X(0,n),\ f=(0=V_0\subset V_1\subset V_2=\bfV)\text{ is a filtration of }\bbN I\text{-graded spaces},\ x(V_i)\subset V_i,\\
&M_{x|_{V_1}}\cong P_{s''}\oplus H_{r_{n-s''}},\ s''<s,\ M_{x|_{V_2/V_1}}\cong S_1,\\
&H_{r_{n-s''}}\text{ is a direct sum of indecomposable regular modules}
\end{aligned}\right\},
\]
which is isomorphic to
\begin{equation}\label{1}
\bigsqcup_{s''<s}\left\{(x,f)\ \middle|\ 
\begin{aligned}
&x\in X(0,n),\ f=(0=V_0\subset V_1\subset V_2\subset V_3=\bfV)\text{ is a filtration of }\bbN I\text{-graded spaces},\ x(V_i)\subset V_i,\\
&M_{x|_{V_1}}\cong H_{r_{n-s''}},\ M_{x|_{V_2/V_1}}\cong P_{s''},\ M_{x|_{V_3/V_2}}\cong S_1,\\
&H_{r_{n-s''}}\text{ is a direct sum of indecomposable regular modules}
\end{aligned}\right\}.
\end{equation}
Fixed $s'<s,$
\begin{equation}
\bigsqcup_{s''<s'}\left\{(x,f)\ \middle|\ 
\begin{aligned}
&x\in X(0,n),\ f=(0=V_0\subset V_1\subset V_2\subset V_3=\bfV)\text{ is a filtration of }\bbN I\text{-graded spaces},\ x(V_i)\subset V_i,\\
&M_{x|_{V_1}}\cong H_{r_{n-s''}},\ M_{x|_{V_2/V_1}}\cong P_{s''},\ M_{x|_{V_3/V_2}}\cong S_1,\\
&H_{r_{n-s''}}\text{ is a direct sum of indecomposable regular modules}
\end{aligned}\right\}
\end{equation} is a closed subvariety of the variety above.
Then $L$ is isomorphic to $\widetilde{L_{(n)}}$, which is an open subset of $L_{(n)}$.
\end{proof}

\begin{lemma}\label{2.5}
For $\nu=(n,n)\in \bbN I$ with $\dim \bfV=\nu$ and $s+s'-1=n$, the variety
\[
L=\left\{(x,f)\ \middle|\ 
\begin{aligned}
&x\in \bfE_{\bfV},\ f=(0=V_0\subset V_1\subset V_2=\bfV)\text{ is a filtration of }\bbN I\text{-graded spaces},\\
&x(V_i)\subset V_i,\ M_{x|_{V_1}}\cong P_{s},\ M_{x|_{V_2/V_1}}\cong I_{s'}
\end{aligned}\right\}
\]
is isomorphic to $K\sqcup U$, where $K$ is closed in this variety and admits a locally trivial fibration over
\[
\left\{x\in \bfE_{\bfV}\ \middle|\ M_x\cong P_{s}\oplus I_{s'} \right\}.
\]
Moreover, for $\underline{\lambda}=(\lambda_1,\ldots,\lambda_k)$, there exists $W'_{0,\underline{\lambda}}$ which admits a locally trivial fibration with projective-space fiber over $W_{0,\underline{\lambda}}$, and there is a locally closed embedding
\[
\widetilde{W'_{0,\underline{\lambda'}}}\hookrightarrow {W'_{0,\underline{\lambda}}}
\quad\text{if }\ \lambda_i=\lambda'_i\ (i\le k-1),\ \lambda_k=\lambda'_k+\lambda'_{k+1}.
\]
Here $\widetilde{W'_{0,\underline{\lambda}}}$ is defined inductively by
\[
\widetilde{W'_{0,\underline{\lambda}}}
=
W'_{0,\underline{\lambda}}
-
\bigsqcup_{\lambda_i=\lambda'_i,\ i\le k-1,\ \lambda_k=\lambda'_k+\lambda'_{k+1}}
\widetilde{W'_{0,\underline{\lambda'}}},
\]
and $U=\widetilde{W'_{0,(n)}}$.
\end{lemma}

\begin{proof}
If $0\rightarrow P_{s}\rightarrow M\rightarrow I_{s'}\rightarrow 0$ is a short exact sequence of representations of $Q$, then either $M\cong P_{s}\oplus I_{s'}$, or $M$ is a direct sum of regular indecomposable representations.

It is straightforward to see that
\[
\left\{(x,f)\ \middle|\ 
\begin{aligned}
&M_x\cong P_{s}\oplus I_{s'},\ f=(0=V_0\subset V_1\subset V_2=\bfV)\text{ is a filtration of }\bbN I\text{-graded spaces},\\
&x(V_i)\subset V_i,\ M_{x|_{V_1}}\cong P_{s},\ M_{x|_{V_2/V_1}}\cong I_{s'}
\end{aligned}\right\}
\]
admits a locally trivial fibration over $\{x\in \bfE_{\bfV}\mid M_x\cong P_{s}\oplus I_{s'}\}$, which is a closed subvariety of $L$. Hence it suffices to consider the case where $M$ is a direct sum of regular indecomposable representations.

In this case, $U$ is isomorphic to the complement in
\[
\left\{(x,f)\ \middle|\ 
\begin{aligned}
&x\in X(0,n),\ f=(0=V_0\subset V_1\subset V_2=\bfV)\text{ is a filtration of }\bbN I\text{-graded spaces},\\
&x(V_i)\subset V_i,\ M_{x|_{V_2/V_1}}\cong I_{s'}
\end{aligned}\right\}
\]
of the closed subvariety
\[
\left\{(x,f)\ \middle|\ 
\begin{aligned}
&x\in X(0,n),\ f=(0=V_0\subset V_1\subset V_2=\bfV)\text{ is a filtration of }\bbN I\text{-graded spaces},\\
&x(V_i)\subset V_i,\ M_{x|_{V_1}}\not\cong P_{s},\ M_{x|_{V_2/V_1}}\cong I_{s'}
\end{aligned}\right\}.
\]
The closed subvariety above is isomorphic to
\begin{equation}\label{eqq}
\bigsqcup_{s''<s}\left\{(x,f)\ \middle|\ 
\begin{aligned}
&x\in X(0,n),\ f=(0=V_0\subset V_1\subset V_2\subset V_3=\bfV)\text{ is a filtration of }\bbN I\text{-graded spaces},\ x(V_i)\subset V_i,\\
&M_{x|_{V_1}}\cong H_r,\ M_{x|_{V_2/V_1}}\cong P_{s''},\ M_{x|_{V_3/V_2}}\cong I_{s'},\\
&H_r\text{ is a direct sum of indecomposable regular modules}
\end{aligned}\right\}.
\end{equation}
Moreover, fixed $s'<s,$
\begin{equation}
\bigsqcup_{s''<s'}\left\{(x,f)\ \middle|\ 
\begin{aligned}
&x\in X(0,n),\ f=(0=V_0\subset V_1\subset V_2\subset V_3=\bfV)\text{ is a filtration of }\bbN I\text{-graded spaces},\ x(V_i)\subset V_i,\\
&M_{x|_{V_1}}\cong H_r,\ M_{x|_{V_2/V_1}}\cong P_{s''},\ M_{x|_{V_3/V_2}}\cong I_{s'},\\
&H_r\text{ is a direct sum of indecomposable regular modules}
\end{aligned}\right\}
\end{equation}
is a closed subvariety of the variety above.
For $\underline{\lambda}=(\lambda_1,\ldots,\lambda_k)$, define the variety ${}^{pre}W'_{0,\underline{\lambda}}$ by
\[
\left\{(x,f,g)\ \middle|\ 
\begin{aligned}
&x\in X(0,n),\ f=(0=V_0\subset V_1\subset \cdots \subset V_{k-1}\subset V_k=\bfV)\text{ is a filtration of }\bbN I\text{-graded spaces},\\
&x(V_i)\subset V_i,\ M_{x|_{V_i/V_{i-1}}}\cong H_{r_i},\ 0\not=g\in \Hom(M_{x|_{V_k/V_{k-1}}},I_{s'}),\\
&H_{r_i}\text{ is a direct sum of indecomposable regular modules, }\dim(H_{r_i})=(\lambda_i,\lambda_i),\ i=1,\ldots,k
\end{aligned}\right\}.
\]
The variety $W^{sur}_{0,\underline{\lambda}}$ is defined as
\[
\left\{(x,f,g)\ \middle|\ 
\begin{aligned}
&x\in X(0,n),\ f=(0=V_0\subset V_1\subset \cdots \subset V_{k-1}\subset V_k=\bfV)\text{ is a filtration of }\bbN I\text{-graded spaces},\\
&x(V_i)\subset V_i,\ M_{x|_{V_i/V_{i-1}}}\cong H_{r_i},\ 0\not=g\in \Hom(M_{x|_{V_k/V_{k-1}}},I_{s'})\text{ is surjective},\\
&H_{r_i}\text{ is a direct sum of indecomposable regular modules, }\dim(M_{r_i})=(\lambda_i,\lambda_i),\ i=1,\ldots,k
\end{aligned}\right\}.
\]
Inductively, set the open subset of ${}^{pre}W'_{0,\underline{\lambda}}$
\[
\widetilde{{}^{pre}W'_{0,\underline{\lambda}}}
=
{}^{pre}W'_{0,\underline{\lambda}}
-
\bigsqcup_{\lambda_i=\lambda'_{i},\ i\le k-1,\ \lambda_k=\lambda'_{k+1}+\lambda'_{k}}
\widetilde{{}^{pre}W'_{0,\underline{\lambda'}}},
\]
and the open subset of $W^{sur}_{0,\underline{\lambda}}$
\[
\widetilde{W^{sur}_{0,\underline{\lambda}}}
=
W^{sur}_{0,\underline{\lambda}}
-
\bigsqcup_{\lambda_i=\lambda'_{i},\ i\le k-1,\ \lambda_k=\lambda'_{k+1}+\lambda'_{k}}
\widetilde{W^{sur}_{0,\underline{\lambda'}}},
\]
where the embeddings and the stratification come from the dimension of the kernel of map $g$ and equation~\ref{eqq} with Lemma~\ref{stra}.

By induction, we obtain \[\begin{aligned}
\widetilde{{}^{pre}W'_{0,\underline{\lambda}}}=&
\left\{(x,f,g)\ \middle|\ 
\begin{aligned}
&x\in \bfE_{\bfV},\ f=(0=V_0\subset V_1\subset \cdots \subset V_{k-1}\subset V_k=\bfV)\text{ is a filtration of }\bbN I\text{-graded spaces},\\
&x(V_i)\subset V_i,\ x(V')\subset V',\ M_{x|_{V_i/V_{i-1}}}\cong H_{r_i},\ \\
&0\not=g\in \Hom(M_{x|_{V_k/V_{k-1}}},I_{s'})\text{ is injective not surjective},\\
&H_{r_i}\text{ is a direct sum of indecomposable regular modules of dimension }(\lambda_i,\lambda_i),\ i=1,\ldots,k
\end{aligned}\right\}\\
\bigcup& \widetilde{W^{sur}_{0,\underline{\lambda}}}.\end{aligned}
\]
In fact, one of the two sets above is empty.

Indeed, when $\underline{\lambda}=(n),$ there does not exist $g\in \Hom(x,I_s)$ injective for $x\in X(0,n),$ $\widetilde{W^{sur}_{0,(n)}}=\widetilde{{}^{pre}W'_{0,(n)}}.$ And $\widetilde{W^{sur}_{0,(n)}}/k^*=U.$

By the same argument as in Lemma~\ref{2.4}, $W'_{0,\underline{\lambda}}:={}^{pre}W'_{0,\underline{\lambda}}/k^*$ admits a locally trivial fibration with projective fiber over $W_{0,\underline{\lambda}}$. Moreover, $W'_{0,\underline{\lambda}}$ is irreducible.
\end{proof}

\begin{lemma}\cite[Theorem 2.2.1]{Jouanolou}\label{aiq}
Let $f:X\to Y$ be a morphism of varieties over $\bbF_q$. If $f$ is a locally trivial fibration with fiber $\bbP^n$, then $\mathbf{H}^i(f_!\overline{\bbQ_l}|_{X})$ is a constant sheaf up to Tate twist.
\end{lemma}

\begin{proof}
First, $\mathbf{H}^i(f_!\overline{\bbQ_l}|_{X})$ is a local system. Hence it suffices to prove that the structure group acts as $\operatorname{Gal}(\overline{\bbF_q}/\bbF_q)$ on $\mathbf{H}^j(\bbP^n,\overline{\bbQ_l})$. Since we can choose an open covering $\cup U=Y$ and $f|_{U}$ as $U\times \bbP^n\rightarrow U.$ We just need to consider the gluing on these $U.$

We know that $\Aut(\bbP^n)\cong \operatorname{PGL}_{n+1}$ by~\cite[Section~0.5]{1994Geometric} and by~\cite[Proposition~10.1]{1980Etale} that
\[
\mathbf{H}^*(\bbP^n,\overline{\bbQ_l})=\overline{\bbQ_l}[h]/(h^{n+1}).
\]
Thus it suffices to show that $\operatorname{PGL}_{n+1}$ acts trivially on $h$. Since $h$ comes from the first Chern class of $\cO_{\bbP^n}(1)$, it is enough to note that for any $g\in \operatorname{PGL}_{n+1}$, the pullback preserves $\cO_{\bbP^n}(1)$. This proves the claim.
\end{proof}

\begin{lemma}\label{buh}
 Let $m,n,t$ be such that $\dim I_m+\dim I_n=2\dim I_t,$ variety $$A_{n,m}=\left\{(x,f)\middle|\begin{aligned}
 &x\in \bfE_{\bfV},M_x\cong I_t\oplus I_t, f=(0\subset V_1\subset V_2=\bfV)\\
 &x(V_i)\subset V_i, M_{x|_{V_1}}\cong I_m, M_{x|_{V_2/V_1}}\cong I_n.
 \end{aligned}\right\},$$ admits a stratification $R=\sqcup U_i,$ such that on each $U_i,$ $U_i\xrightarrow{p_i} \bfE_{\bfV},$ factors through $$A_{t,t}=\left\{(x,f)\middle|\begin{aligned}
 &x\in \bfE_{\bfV},M_x\cong I_t\oplus I_t, f=(0\subset V_1\subset V_2=\bfV)\\
 &x(V_i)\subset V_i, M_{x|_{V_1}}\cong I_t, M_{x|_{V_2/V_1}}\cong I_t.
 \end{aligned}\right\},$$
\end{lemma}
\begin{proof}
 Fix $x_0$ such that $M_{x_0}\cong I_t\oplus I_t,$ and denote its automorphism group by $H=\bfG l(2,k).$ Let $$A'_{n,m}=\left\{(x_0,f)\middle|\begin{aligned}
 & f=(0\subset V_1\subset V_2=\bfV)\\
 &x_0(V_i)\subset V_i, M_{x_0|_{V_1}}\cong I_m, M_{x_0|_{V_2/V_1}}\cong I_n.
 \end{aligned}\right\},$$ Then $A_{n,m}\cong A'_{n,m}\times_{H} \bfG_{\bfV}.$ Any $H$-stable stratification on $A'_{n,m}$ induces $A_{n,m}.$ Also fix $\tilde{V}\subset \bfV$ such that $M_{x_0|_{\tilde{V}}}\cong I_t.$

 For $0\rightarrow I_m\xrightarrow{f} I_t\oplus I_t,$ we have $\dim \operatorname{coker} f=\dim I_n,$ and $\operatorname{coker} f$ is preinjective, thus $\operatorname{coker} f\cong I_n.$
 Thus for $A'_{n,m},$ we only need to fix $x_m$ such that $M_{x_m}\cong I_m,$ then $A'_{n,m}\cong \{f\in \Hom_{kQ}(M_{x_m},M_{x_0})/k^*|\ker f=0\},$ which is an open subset in $(\Hom_{kQ}(M_{x_m},M_{x_0})-{0})/k^*.$ Since $(\Hom_{kQ}(M_{x_m},M_{x_0})-{0})/k^*\cong (\Hom_{kQ}(I_m,I_t)\oplus \Hom_{kQ}(I_m,I_t)-{0})/k^*,$ if we denote $C=\Hom_{kQ}(x_m,x_0),$ $A'_{n,m}$ is an open subvariety in $\{[a,b]\in C\oplus C-{0}/k^*|a,b\text{ are linear independent}\}$ with $H=\bfG l(2,k)$ action. $C$ is $m-t+1$ dimensional $k$-space.
 Because $C\oplus C\cong \Hom_k(k^2,C),$ in our case, we only consider the subvariety $\Hom^0_k(k^2,C)=\{f\in \Hom_k(k^2,C)|rank(f)=2\}.$ $\Hom^0_k(k^2,C)/k^*$ has a canonical map $\pi$ to Grassmannian $Gr(2,C).$ For the Grassmannian $Gr(2,C),$ we give a natural stratification.
 
 Actually, when we give $C$ a $k-$basis $e_1,\cdots,e_{m-t+1},$ the following locally closed subvariety, for $0<a<b\leq m-t+1,$ $U'_{a,b}=\{[x,y]\in Gr(2,C)|x=e_a+\sum_{j<a} \alpha_j e_j, y=e_b+\sum_{j'<b,\not=a}\beta_{j'}e_{j'}, \alpha_j,\beta_{j'}\in k\}$ which is a Schubert cell, gives the desired stratification. We use the stratification $\tilde{U}_{a,b}=\pi^{-1}(U'_{a,b})\cap A'_{n,m}.$ It is clearly stable under $H.$ For $0<a<b\leq m-t+1,$ denote the subvariety $\{[x,y]\in (C\oplus C-{0})/k^*|x=e_a+\sum_{j<a} \alpha_j e_j, y=e_b+\sum_{j'<b,\not=a}\beta_{j'}e_{j'}, \alpha_j,\beta_{j'}\in k\}$ as $L_{a,b}$ and $L'_{a,b}=A'_{n,m}\cap L_{a,b}.$ Then $\tilde{U}_{a,b}\cong L'_{a,b}\times PGl(2,k).$

 The stratification $\tilde{U}_{a,b}$ induces the stratification $U_{a,b}\cong L'_{a,b}\times (G_{\bfV}/k^*)$ on $A_{n,m}'.$ Then we define a morphism from $U_{a,b}$ to $A_{t,t}$ as follows.
 $$\begin{tikzcd}
{L'_{a,b}\times (G_{\bfV}/k^*)} \arrow[r, "{p_{a,b}}"] & { A_{t,t}} \\
{((x_0,f),g)} \arrow[r, maps to] & {(g(x_0),g(\tilde{V}))}
\end{tikzcd}$$
It is straightforward to see that $U_{a,b}\rightarrow \bfE_{\bfV}$ factors through $A_{t,t}$ by $U_{a,b}\xrightarrow{p_{a,b}}A_{t,t}\rightarrow \bfE_{\bfV}.$ This completes the proof. Indeed, we can choose $x_m,$ $(x_0, \tilde{V}),$ and $e_i$ above $\operatorname{Fr^r}-$fixed.
\end{proof}
Furthermore, $A_{t,t}$ is $\bfG_{\bfV}$-orbit with connected automorphism group.

For the following proof, we recall some properties of the derived category of sheaf complexes.

Denote the perverse $t$-structure truncation functors by ${}^p\tau^{\geq i}$ and ${}^p\tau^{\leq i}$, and the perverse cohomology functors by ${}^p\mathbf{H}^i(\ \cdot\ )$. 
\begin{lemma}\cite[Theorem 1.3.10]{Pramod-2021}\label{ar}
For a variety $X,$ $U\xrightarrow{j} X$ is an open embedding and $Z\xrightarrow{i} X$ is the complement of $U$. Then if $\mathcal{F}\in \cD_{c}^b(X)$, there is a distinguished triangle $$j_!j^*\mathcal{F}\rightarrow \mathcal{F}\rightarrow i_!i^*\mathcal{F}''\xrightarrow[]{+1}.$$

For a variety $X$, if $\mathcal{F}\in \cD_{c}^b(X)$, and there is a distinguished triangle
\[
\mathcal{F}'\rightarrow \mathcal{F}\rightarrow \mathcal{F}''\xrightarrow[]{+1}.
\]
Then there is a long exact sequence of perverse sheaves
\[
\cdots\rightarrow {}^p\mathbf{H}^r(\mathcal{F}')\rightarrow {}^p\mathbf{H}^r(\mathcal{F})\rightarrow {}^p\mathbf{H}^r(\mathcal{F}'')\xrightarrow[]{+1} {}^p\mathbf{H}^{r+1}(\mathcal{F}')\rightarrow \cdots.
\]
\end{lemma}
Thus, each simple constituent of ${}^p\mathbf{H}^r(\mathcal{F})$ is a simple constituent of either ${}^p\mathbf{H}^r(\mathcal{F}')$ or ${}^p\mathbf{H}^r(\mathcal{F}'')$.

Here we start the proof of Theorem~\ref{th}.

\begin{proof}
 We consider the sequence $\hat{\underline{\nu}}=(\hat{\nu_1},\cdots,\hat{\nu_n}),$ where $\sum \hat{\nu_i}=\nu,$ and for each $\hat{\nu}_i=((\hat{\nu}_i)_1,(\hat{\nu}_i)_2),$ we assume $|(\hat{\nu}_i)_1-(\hat{\nu}_i)_2)|\leq 1.$ We denote the set consisting of such sequences $\mathcal{S}q_{\nu}.$ Then there is a fixed $\bfE_{\hat{\nu}_i}^0\subset \bfE_{\bfV_i},\dim \bfV_i=\hat{\nu}_i$, where $\bfE_{\hat{\nu}_i}^0$ is the orbit of indecomposable preprojective representation, indecomposable preinjective representation, or the subvariety consisting of regular representations. 
 Consider $$S_{\underline{\hat{\nu}}}=\left\{(x\in \bfE_{\bfV},f)\middle|\begin{aligned}
 &f=(0=V_0\subset\cdots\subset V_{k-1}\subset V_k\subset V_{k+1}\subset\cdots\subset V_{n}=\bfV, \dim (V_i/V_{i-1})=\hat{\nu}_i,\\
 &x|_{V_i/V_{i-1}}\in \bfE_{\hat{\nu}_i}^0.
 \end{aligned}\right\}.$$
 Claim: To study the simple perverse constituents, assume we start with $\overline{\bbQ_l}|_{S_{\underline{\hat{\nu}}}}\boxtimes\overline{\bbQ_l}(\zeta)$ on some $S_{\underline{\hat{\nu}}}$, where $\overline{\bbQ_l}(\zeta)$ is pure on $\operatorname{Spec}(\bbF_q)$ and $\operatorname{egf}(\overline{\bbQ_l}(\zeta))=\overline{\bbq_l}.$ All the simple perverse constituents come from ${p_{\alpha',\underline{\lambda'}}}_!\overline{\bbQ_l}|_{W_{\alpha',\underline{\lambda'}}}\boxtimes\overline{\bbQ_l}(\zeta),$ for $\overline{\bbQ_l}(\zeta)$ above. $-\boxtimes \overline{\bbQ_l}(\zeta)$ is exact and perverse $t$-exact with $f_!(\mathcal{F})\boxtimes \overline{\bbQ_l}(\zeta)\cong (f_!\mathcal{F})\boxtimes \overline{\bbQ_l}(\zeta)$ and $g^*(\mathcal{F}\boxtimes \overline{\bbQ_l}(\zeta))\cong (g^*\mathcal{F})\boxtimes \overline{\bbQ_l}(\zeta).$ Thus we may start with the constant sheaf $\overline{\bbQ_l}|_{S_{\underline{\hat{\nu}}}}.$
 
 If there exists $\hat{\nu}_k=(s'-1,s')$ and $\hat{\nu}_{k+1}=(s,s-1)$ in $\underline{\hat{\nu}},$
 we fix $\bfV_{k-1}\subset \bfV_{k+1}$ with dimensions $\sum_{j=1}^{k-1}\hat{\nu}_j,\sum_{j=1}^{k+1}\hat{\nu}_j$, and the stabilizer $P\subset \bfG_{\bfV}$ of $\bfV_{k+1}$ and $\bfV_{k-1}.$ Also denote the unipotent radical of $P$ as $U.$

 For $$\left\{(x\in \bfE_{\bfV},f)\middle|\begin{aligned}
 &f=(0=V_0\subset \cdots\subset V_{k-1}\subset V_k\subset V_{k+1}\subset\cdots \subset V_{n}=\bfV), x(V_i)\subset V_i,\\
 &\dim (V_i/V_{i-1})=\hat{\nu}_i,x|_{V_i/V_{i-1}}\in \bfE_{\hat{\nu}_i}^0,i\not=k,k+1,\\
 &M_{x|_{V_{k}/V_{k-1}}}\cong P_{s'}, M_{x|_{V_{k+1}/V_k}}\cong I_{s}.
 \end{aligned}\right\}$$ is isomorphic to 
 \begin{equation}\label{con} 
 \bfG_{\bfV}\times_{P}\left\{(x\in \bfE_{\bfV},f)\middle|\begin{aligned}
 &f=(0=V_0\subset\cdots \subset\bfV_{k-1}=V_{k-1}\subset V_k\subset \bfV_{k+1}=V_{k+1}\subset\cdots\subset V_{n}=\bfV),x(V_i)\subset V_i,,\\
 &\dim (V_i/V_{i-1})=\hat{\nu}_i,x|_{V_i/V_{i-1}}\in \bfE_{\hat{\nu}_i}^0,i\not=k,k+1,\\
 &M_{x|_{V_{k}/\bfV_{k-1}}}\cong P_{s'}, M_{x|_{\bfV_{k+1}/V_k}}\cong I_{s},x(\bfV_{k+1})\subset \bfV_{k+1},x(\bfV_{k-1})\subset \bfV_{k-1}.
 \end{aligned}\right\}.\end{equation}
 We denote $$F=\left\{(x\in \bfE_{\bfV},f)\middle|\begin{aligned}
 &f=(0=V_0\subset\cdots \subset\bfV_{k-1}=V_{k-1}\subset V_k\subset \bfV_{k+1}=V_{k+1}\subset\cdots\subset V_{n}=\bfV),x(V_i)\subset V_i,\\
 &\dim (V_i/V_{i-1})=\hat{\nu}_i,x|_{V_i/V_{i-1}}\in \bfE_{\hat{\nu}_i}^0,i\not=k,k+1,\\
 &M_{x|_{V_{k}/\bfV_{k-1}}}\cong P_{s'}, M_{x|_{\bfV_{k+1}/V_k}}\cong I_{s},x(\bfV_{k+1})\subset \bfV_{k+1},x(\bfV_{k-1})\subset \bfV_{k-1},.
 \end{aligned}\right\}.$$ 
 If we denote
 
 $$E_1=\left\{(x\in \bfE_{\bfV_{k-1}},f)\middle|\begin{aligned}
 &f=(0=V_0\subset V_1\cdots \subset\bfV_{k-1}),x(V_i)\subset V_i,\\
 &x|_{V_i/V_{i-1}}\in \bfE_{\hat{\nu}_i}^0,
 \end{aligned}\right\},$$ $$E_2=\left\{(x\in \bfE_{\bfV_{k+1}/\bfV_{k-1}},f)\middle|\begin{aligned}
 &f=(0\subset V_k/\bfV_{k-1}\subset \bfV_{k+1}/\bfV_{k-1}),x(V_{k}/\bfV_{k-1})\subset V_{k}/\bfV_{k-1},,\\
 &M_{x|_{V_{k}/\bfV_{k-1}}}\cong P_{s'}, M_{x|_{\bfV_{k+1}/V_k}}\cong I_{s}.
 \end{aligned}\right\}$$ and $$E_3=\left\{(x\in \bfE_{\bfV/\bfV_{k+1}},f)\middle|\begin{aligned}
 &f=(0\subset V_{k+2}/\bfV_{k+1}\cdots\subset V_{n}/\bfV_{k+1}=\bfV/\bfV_{k+1}),x(V_i/\bfV_{k+1})\subset V_i/\bfV_{k+1},\\
 &\dim (V_i/V_{i-1})=\hat{\nu}_i,x|_{V_i/V_{i-1}}\in \bfE_{\hat{\nu}_i}^0,\\
 \end{aligned}\right\}$$
 
 Then for $$E_1\times E_2\times E_3\xleftarrow{p_1} F\times_{U}\bfG_{\bfV}\xrightarrow{p_2} F\times_{P}\bfG_{\bfV},$$ $p_2$ is principal bundle and $p_1$ is smooth with connected fibers. The stratification of $E_1\times E_2\times E_3$ factors through $p_2\circ p_1^{-1}$ gives a stratification on $F\times_{P}\bfG_{\bfV}.$
 
 By Lemma~\ref{2.5}, $E_2$ is isomorphic to $K\sqcup U$ and by Lemma~\ref{aiq}, the trivial fibration morphism $f$ with projective fiber satisfying that $\mathbf{H}^j(f_!\overline{\bbQ_l})$ is constant sheaf with Tate twist. We apply Lemma~\ref{2.5} to $E_2$ and by using $f_!$ for the corresponding projection map $f:W'_{0,\underline{\lambda}}\rightarrow W_{0,\underline{\lambda}}$. For the closed subset $K,$ any $\bfG_{\bfV}$-equivariant simple lisse sheaf over an orbit with integral weights has the form $\overline{\bbq_l}\boxtimes \overline{\bbQ_l}(\zeta),$ where $\overline{\bbQ_l}(\zeta)$ is pure on $\operatorname{Spec}(\bbF_q).$ This is because $\Aut_x$ is connected for any $x\in \bfE_{\bfV}.$ Thus for $K\xrightarrow{f}\mo,$ $\mathbf{H}^i(f_!\overline{\bbQ_l})$ is generated by $\overline{\bbQ_l}|_{\mo}\boxtimes\overline{\bbQ_l}(\zeta),$ for some $\zeta.$ 
 Denote orbit $\mo$ as the orbit of $P_s'\oplus I_s$ in $\bfE_{\bfV_{k+1}/\bfV_{k-1}},$ $$F_\mo=\left\{(x\in \bfE_{\bfV},f)\middle|\begin{aligned}
 &f=(0=V_0\subset\cdots \subset\bfV_{k-1}=V_{k-1}\subset V_{k+1}=\bfV_{k+1}\subset\cdots\subset V_{n}=\bfV),x(V_i)\subset V_i,\\
 &\dim (V_i/V_{i-1})=\hat{\nu}_i,x|_{V_i/V_{i-1}}\in \bfE_{\hat{\nu}_i}^0,i\not=k,k+1,\\
 &M_{x|_{\bfV_{k+1}/\bfV_{k-1}}}\cong I_{s}\oplus P_{s'}.
 \end{aligned}\right\},$$
 $$F_K=\left\{(x\in \bfE_{\bfV},f)\middle|\begin{aligned}
 &f=(0=V_0\subset\cdots \subset\bfV_{k-1}=V_{k-1}\subset V_k\subset \bfV_{k+1}=V_{k+1}\subset\cdots\subset V_{n}=\bfV),x(V_i)\subset V_i,,\\
 &\dim (V_i/V_{i-1})=\hat{\nu}_i,x|_{V_i/V_{i-1}}\in \bfE_{\hat{\nu}_i}^0,i\not=k,k+1,\\
 &M_{x|_{\bfV_{k+1}/\bfV_{k-1}}}\cong I_{s}\oplus P_{s'},M_{x|_{V_{k}/\bfV_{k-1}}}\cong P_{s'}, M_{x|_{\bfV_{k+1}/V_k}}\cong I_{s}.
 \end{aligned}\right\},$$
 Consider 
 $$\begin{tikzcd}
E_1\times K\times E_3 \arrow[d, "\pi_1"] & F_K\times_U \bfG_{\bfV} \arrow[l, "p_1"'] \arrow[r, "p_2"] \arrow[d] & F_K\times_P \bfG_{\bfV} \arrow[d, "\pi_2"] \\
E_1\times \mo\times E_3 & F_{\mo}\times_U \bfG_{\bfV} \arrow[l, "p_1"'] \arrow[r, "p_2"] & F_{\mo}\times_P \bfG_{\bfV} 
\end{tikzcd}$$ we have that ${\pi_2}_!{p_2}_{\flat}p_1^*(\overline{\bbQ_l})\cong {p_2}_{\flat}p_1^*{\pi_1}_!(\overline{\bbQ_l}),$ hence $\mathbf{H}^i({\pi_2}_!{p_2}_{\flat}p_1^*(\overline{\bbQ_l}))\cong \mathbf{H}^i({p_2}_{\flat}p_1^*{\pi_1}_!(\overline{\bbQ_l}))\cong {p_2}_{\flat}p_1^*\mathbf{H}^i({\pi_1}_!(\overline{\bbQ_l})).$ We already have that $\mathbf{H}^i({\pi_1}_!(\overline{\bbQ_l}))$ is generated by $\overline{\bbQ_l}|_{E_1\times \mo\times E_3}\boxtimes\overline{\bbQ_l}(\zeta),$ for some $\zeta.$ Furthermore, $p_1^*$ and ${p_2}_{\flat}$ are exact functors, thus by Lemma~\ref{ar} $\mathbf{H}^i({p_2}_{\flat}p_1^*{\pi_1}_!(\overline{\bbQ_l}))$ is generated by $\overline{\bbQ_l}|_{F_{\mo}\times_P \bfG_{\bfV}}\boxtimes\overline{\bbQ_l}(\zeta),$ for some $\zeta.$ 

If we denote 
$$F_{W_{0,\underline{\lambda}}}=\left\{(x\in \bfE_{\bfV},f)\middle|\begin{aligned}
 &f=(0=V_0\subset\cdots \subset V_{k-1}=\bfV_{k-1}=V_k^0\subset\cdots\subset V_k^{t-1}\subset V_k^t=\bfV_{k+1}=V_{k+1}\cdots\subset V_{n}=\bfV),\\
 &x(V_i)\subset V_i,\dim (V_i/V_{i-1})=\hat{\nu}_i,x|_{V_i/V_{i-1}}\in \bfE_{\hat{\nu}_i}^0,i\not=k,k+1, x(V_k^i)\subset V_k^i,\\
 &M_{x|_{V_{k}^i/V_{k}^{i-1}}} \text{ is regular of dimension }(\lambda_i,\lambda_i).
 \end{aligned}\right\}$$
and
$$F_{W'_{0,\underline{\lambda}}}=\left\{(x\in \bfE_{\bfV},f,g)\middle|\begin{aligned}
 &f=(0=V_0\subset\cdots \subset V_{k-1}=\bfV_{k-1}=V_k^0\subset\cdots\subset V_k^{t-1}\subset V_k^t=\bfV_{k+1}=V_{k+1}\cdots\subset V_{n}=\bfV),\\
 &x(V_i)\subset V_i,\dim (V_i/V_{i-1})=\hat{\nu}_i,x|_{V_i/V_{i-1}}\in \bfE_{\hat{\nu}_i}^0,i\not=k,k+1, x(V_k^i)\subset V_k^i,\\
 &M_{x|_{V_{k}^i/V_{k}^{i-1}}} \text{ is regular of dimension }(\lambda_i,\lambda_i),\\
 &0\not=g\in \Hom(M_{x|_{V_k^t/V_{k}^{t-1}}},I_{s'})/k^*.
 \end{aligned}\right\}$$ $r\in P$ acts on $F_{W'_{0,\underline{\lambda}}}$ as $r(x,f,g)=(rxr^{-1},rf,gr^{-1})$ where for $g\in \Hom(M_{x|_{V_k^t/V_{k}^{t-1}}},I_{s'})/k^*,$ $gr^{-1}\in \Hom(M_{x|_{r(V_k^t)/r(V_{k}^{t-1})}},I_{s'})/k^*.$
 
 By the same stratification in Lemma~\ref{2.5}, we only need to consider $\bfG\times_{P}F_{W'_{0,\underline{\lambda}}}\xrightarrow{\pi} \bfE_{\bfV},$ which factors through $\bfG\times_{P}F_{W'_{0,\underline{\lambda}}}\xrightarrow{\pi_{\underline{\lambda}}}\bfG\times_{P}F_{W_{0,\underline{\lambda}}}\xrightarrow{\pi'} \bfE_{\bfV}.$ And $\pi_{\underline{\lambda}}$ is a locally trivial fibration with fiber $\bbP_{\lambda_t-1},$ thus by Lemma~\ref{aiq}, $\mathbf{H}^i({\pi_{\underline{\lambda}}}_!(\overline{\bbQ_l}))$ has the form $\overline{\bbQ_l}(-\frac{i}{2}).$ 
 
 Then, by Lemma~\ref{ar}, we know that we only need to consider sheaf $\overline{\bbQ_l}$ on $ 
 \bfG_{\bfV}\times_{P} F_{\mo}$ and 
 $ 
 \bfG_{\bfV}\times_{P}F_{W_{0,\underline{\lambda}}}$ which both have the form $S_{\hat{\underline{\nu}}'},$ for some $\hat{\underline{\nu}}'\in \mathcal{S}q_{\nu}.$ 
 
 If there exists $\hat{\nu_k}=(r,r),$ $\hat{\nu}_{k+1}=(s,s-1),$ or $\hat{\nu_k}=(s-1,s),$ $\hat{\nu}_{k+1}=(r,r),$ this is the case of Lemma~\ref{2.4} and its preprojective analogue, the proof is the same. If $\hat{\nu_k}=(s',s'-1),$ $\hat{\nu}_{k+1}=(s,s-1),$ with $s'>s$ or $\hat{\nu_k}=(s'-1,s'),$ $\hat{\nu}_{k+1}=(s-1,s),$ with $s'<s,$ by Lemma~\ref{buh}, we only need to consider $$F_{a,b}=\left\{(x\in \bfE_{\bfV},f)\middle|\begin{aligned}
 &f=(0=V_0\subset\cdots \subset\bfV_{k-1}=V_{k-1}\subset V_k\subset \bfV_{k+1}=V_{k+1}\subset\cdots\subset V_{n}=\bfV),x(V_i)\subset V_i,\\
 &\dim (V_i/V_{i-1})=\hat{\nu}_i,x|_{V_i/V_{i-1}}\in \bfE_{\hat{\nu}_i}^0,i\not=k,k+1,\\
 &M_{x|_{\bfV_{k+1}/\bfV_{k-1}}}\cong I_{t}\oplus I_{t},(x,f)|_{{\bfV_{k+1}/\bfV_{k-1}}}\in U_{a,b}.
 \end{aligned}\right\},$$
 
 and $$F_t=\left\{(x\in \bfE_{\bfV},f)\middle|\begin{aligned}
 &f=(0=V_0\subset\cdots \subset\bfV_{k-1}=V_{k-1}\subset V_k\subset \bfV_{k+1}=V_{k+1}\subset\cdots\subset V_{n}=\bfV),x(V_i)\subset V_i,\\
 &\dim (V_i/V_{i-1})=\hat{\nu}_i,x|_{V_i/V_{i-1}}\in \bfE_{\hat{\nu}_i}^0,i\not=k,k+1,\\
 &M_{x|_{\bfV_{k+1}/\bfV_{k-1}}}\cong I_{t}\oplus I_{t},(x,f)|_{{\bfV_{k+1}/\bfV_{k-1}}}\in A_{t,t}.
 \end{aligned}\right\}.$$ 
 
 Denote $\operatorname{Pr}:\bfV_{k+1}\rightarrow \bfV_{k+1}/\bfV_{k-1}.$ And morphism from $F_{a,b}\xrightarrow{P_{a,b}} F_t$ induced by $p_{a,b}=(p^1_{a,b},p_{a,b}^2)$ in Lemma~\ref{buh}, where $p^1_{a,b}$ gives part in $\bfE_{\bfV_{k+1}/\bfV_{k-1}}$ and $p^2_{a,b}$ gives the Grassmannian, is defined as $$\begin{aligned}
 &P_{a,b}(x,(0=V_0\subset\cdots \subset\bfV_{k-1}=V_{k-1}\subset V_k\subset \bfV_{k+1}=V_{k+1}\subset\cdots\subset V_{n}=\bfV))\\&=(x,(0=V_0\subset\cdots \subset\bfV_{k-1}=V_{k-1}\subset \operatorname{Pr}^{-1}(p^2_{a,b}((x,f)|_{\bfV_{k+1}/\bfV_{k-1}}))\subset \bfV_{k+1}=V_{k+1}\subset\cdots\subset V_{n}=\bfV)).
 \end{aligned}$$ It is $P$-equivariant morphism. Moreover, $A_{t,t}$ is $\bfG_{\bfV_{k+1}/\bfV_{k-1}}$-orbit, then, using the same argument as for $\mo$ and $K,$ we only need to study $F_t\times_P\bfG_{\bfV},$ which also has the form $S_{\hat{\underline{\nu''}}},$ for some $\hat{\underline{\nu''}}\in \mathcal{S}q_{\nu}.$ If the middle term has the form $I_{t'}\oplus I_{t''}, t'\not=t'',$ the proof is the same as dealing with $\mo$ and $K.$ 
 
 Since no new $I_{s'}$ or $P_{s'}$ with smaller dimension appears in the above steps, after finitely many steps, we only need to consider the following case.

Fix $(s_i)_{i=1,\cdots,k}$ and $(s'_i)_{i=1,\cdots,l}$ such that $i\leq j$, $s_i\leq s_j$, and $i\leq j$, $s'_i\geq s'_j$,
 $$W'=\left\{(x,f)\middle|\begin{aligned}
 &x\in \bfE_{\bfV}, f=(0=V_0\subset \cdots\subset V_k=V_k^0\subset V_k^1\subset\cdots\subset V_k^t=V_{k+1}\subset\cdots V_{k+1+l}=\bfV),\\
 &x(V_i)\subset V_i, x(V_k^j)\subset V_k^j, M_{x|_{V_i/V_{i-1}}}\cong I_{s_i}, i=1,\cdots,k, \\
 & M_{x|_{V_{i+k+1}/V_{i+k}}}\cong P_{s'_i}, i=1,\cdots,l,\\
 &M_{x|_{V_k^r/V_k^{r-1}}}\text{ is regular with dimension }(\lambda_r,\lambda_r).
 \end{aligned}\right\}\xrightarrow{\pi} \bfE_{\bfV}.$$
We group together the identical preprojective and preinjective parts, and for $(s_i),(s_i')$ above, if $(s_i)=(s_1,\cdots,s_1, s_2,\cdots, s_2,\cdots),$ we denote $t_1=s_1,$ $t_2=s_2,\cdots$ with $a_i$ is the number of occurrence of $s_i$ in $(s_i)$ and if $(s'_i)=(s'_1,\cdots,s'_1, s'_2,\cdots, s'_2,\cdots),$ we denote $t'_1=s'_1,$ $t'_2=s'_2,\cdots$ with $b_i$ is the number of occurrence of $s'_i$ in $(s'_i).$ Now with new sequence $(t_i)_{i=1,\cdots,u}$ and $(t'_i)_{i=1,\cdots,v},$ there is $$W=\left\{(x,f)\middle|\begin{aligned}
 &x\in \bfE_{\bfV}, f=(0=V'_0\subset \cdots\subset V'_u=V_u^0\subset {V'_u}^1\subset\cdots\subset {V'_u}^t=V'_{u+1}\subset\cdots \subset V'_{u+v+1}=\bfV),\\
 &x(V'_i)\subset V'_i, x({V'_k}^j)\subset {V'_k}^j, M_{x|_{V'_i/V_{i'-1}}}\cong I_{t_i}^{\oplus a_i}, i=1,\cdots,u, \\
 & M_{x|_{V'_{i+u+1}/V'_{i+u}}}\cong P_{t'_i}^{\oplus b_{i}}, i=1,\cdots,v, \\
 &M_{x|_{{V'_u}^r/{V'_u}^{r-1}}}\text{ is regular with dimension }(\lambda_r,\lambda_r).
 \end{aligned}\right\},$$ $\pi$ factors through $W'\xrightarrow{\pi'}W\xrightarrow{\pi''} \bfE_{\bfV}.$ We only need to study $\pi'_!(\overline{\bbQ_l}).$
 As in the treatment of the orbit $\mo$ and of $K$ above, we only need to consider $\overline{\bbQ_l}|_W\boxtimes \overline{\bbq_l}(\zeta)$ on $W.$ We work in the mixed derived category of sheaf complexes. Thus $\overline{\bbQ_l}(\zeta)\boxtimes\overline{\bbQ}(\zeta')$ appeared above is isomorphic to $\overline{\bbQ_l}(\zeta')$ for some $\zeta'$ such that $\overline{\bbQ_l}(\zeta'')$ is pure on $\operatorname{Spec}(\bbF_q)$ and $\operatorname{egf}(\overline{\bbQ_l}(\zeta''))=\overline{\bbq_l}.$

 After the preceding reductions, we finally obtain varieties of the form $W_{\alpha'',\underline{\lambda}''}$ and consider the constant sheaf on them. It is straightforward to see that
 $$p_{\alpha,\underline{\lambda}}:W_{\alpha,\underline{\lambda}}\rightarrow \bfE_{\bfV}$$ factors through $$p_{\alpha,\underline{\lambda}}:W_{\alpha,\underline{\lambda}}\rightarrow X(\alpha,m)\xrightarrow{j_{\alpha,m}}\bfE_{\bfV}.$$ And by \cite[Lemma 1.3.10(1)]{Pramod-2021}, $j_{\alpha',m'}^*{j_{\alpha,m}}_!=0$ for $(\alpha',m')\not=(\alpha,m).$ 
 By Lemma~\ref{ar}, simple constituents of ${}^p\mathbf{H}^j\!\left({j_{\alpha',m'}}^*{\pi_{\underline{\nu}}}_!\overline{\bbQ_l}\right)$ are among the simple constituents of $^p\mathbf{H}^j({p_{\alpha',\underline{\lambda}'}}_!\overline{\bbQ_l})\boxtimes \overline{\bbQ_l}(\zeta).$ This completes the proof.
\end{proof}
\section{The proof of Proposition~\ref{propp1}}\label{sec2}
To prove Proposition~\ref{propp1}, we show that the proof of Theorem~\ref{th} is compatible with the chosen $\mathbb{F}_q$-structures

\begin{remark}\label{rem}
Since the flag varieties and locally closed subsets in Lemma~\ref{2.4} and \ref{2.5}, $W_{\alpha,\underline{\lambda}},$ $W_{\alpha,\underline{\lambda}}^{rss},$ $X(\alpha,m)$ and $X(\alpha)$ are stable under Frobenius map. $S_m/S_{\lambda_1}\times \cdots\times S_{\lambda_n}$ is compatible with Frobenius map and the map $r_X:X=X_0\times_{\operatorname{Spec}(\bbF_q)}\operatorname{Spec(\overline{\bbF_q})}\rightarrow X_0$ which is induced by $\operatorname{Spec}(\overline{\bbF_q})\rightarrow \operatorname{Spec}(\bbF_q)$ is faithfully flat and quasi-compact (fpqc). Thus our chosen $\bbF_q$-structures, they preserve open and closed embeddings, smoothness, and finite étale covers by \cite[Proposition 2.36]{vistoli2007notesgrothendiecktopologiesfibered}. Thus $\bbF_q$-structures are compatible with Theorem~\ref{th}.\end{remark}

\begin{lemma}\cite[Theorem~5.4.16]{Pramod-2021}
Every mixed perverse sheaf $\mathcal{F}$ is equipped with a canonical filtration
\[
\cdots \subset W_{i-1}\mathcal{F}\subset W_{i}\mathcal{F}\subset W_{i+1}\mathcal{F}\subset \cdots
\]
such that $\operatorname{gr}_a\mathcal{F}=W_a\mathcal{F}/W_{a-1}\mathcal{F}$ is pure of weight $a$. Moreover, every morphism $\phi:\mathcal{F}\rightarrow \mathcal{G}$ is strictly compatible with this filtration.
\end{lemma}

By \cite[Theorem~2.1]{sunarticle}, this lemma is still true for $\bfG_{\bfV}$-equivariant perverse sheaves.

We will consider $$\mathcal{L}=\overline{\bbQ_l}|_{\bfE_{\bfV_{s_1}}}\diamond \cdots \diamond \overline{\bbQ_l}|_{\bfE_{\bfV_{s_r}}}
\diamond \overline{\bbQ_l}|_{\bfE_{\bfV_{\lambda_1}}}\diamond \overline{\bbQ_l}|_{\bfE_{\bfV_{\lambda_2}}}\cdots \diamond \overline{\bbQ_l}|_{\bfE_{\bfV_{\lambda_k}}}
\diamond \overline{\bbQ_l}|_{\bfE_{\bfV'_{s'_1}}}\diamond \cdots \diamond \overline{\bbQ_l}|_{\bfE_{\bfV'_{s'_t}}}$$ from equation~\ref{alpha} satisfying $\lambda_i=1$ which is a direct summand up to shift and Tate twist of the flag sheaf complex $L_{\underline{\nu}},$ where $\underline{\nu}=(\nu_1,\cdots,\nu_k),$ $\nu_i=i_1$ or $i_2.$ $IC(X(\alpha),L_{\chi})$ is a direct summand of $\mathcal{L}.$ Since the flag sheaf complex is defined over $\bbF_q,$ denote the corresponding element as $\mathcal{L}_0,$ $\operatorname{egf}(\mathcal{L}_0)=\mathcal{L}.$ We will prove that $\mathcal{L}_0$ is pointwise pure in the next proposition.
Now recall that $j_{\alpha'}:X(\alpha')\rightarrow \bfE_{\bfV}.$

We now prove Proposition~\ref{propp1}.
\begin{proof}
 We already have that $\supp(\mathcal{L}_0)=\overline{X(\alpha,m)_0}.$

 If $x\in X(\alpha,m)_0(\bbF_{q^l}), \operatorname{Spec}(\bbF_{q^l})\xrightarrow{i_x} \bfE_\bfV,$ since the property of \begin{equation}\label{eq12}
 i_x^*\overline{\bbQ_l}|_{\bfE_{\bfV_{s_1}}}\diamond \cdots \diamond \overline{\bbQ_l}|_{\bfE_{\bfV_{s_r}}}
\diamond \overline{\bbQ_l}|_{\bfE_{\bfV_{\lambda_1}}}\diamond \overline{\bbQ_l}|_{\bfE_{\bfV_{\lambda_2}}}\cdots \diamond \overline{\bbQ_l}|_{\bfE_{\bfV_{\lambda_k}}}
\diamond \overline{\bbQ_l}|_{\bfE_{\bfV'_{s'_1}}}\diamond \cdots \diamond \overline{\bbQ_l}|_{\bfE_{\bfV'_{s'_t}}},\lambda_i=1, 
 \end{equation} is from a map $S\rightarrow \operatorname{Spec(\bbF_{q^l})}$, where $S$ admits an affine paving as in Jordan quiver \cite{articleff}. Thus the proposition is proved for $x\in X(\alpha,m)_0(\bbF_{q^l})$. Furthermore, $${j_{\alpha}}^*\mathcal{L}_0=\oplus_{\chi\in X(S_m)}L_{\chi}^{\oplus s_{\chi}}[\dim X(\alpha)](\frac{\dim X(\alpha)}{2})$$ by the etale covering map, and $\operatorname{egf}({}^p\mathbf{H}^{i}(j_{\alpha,m}^*\mathcal{L}_0))={}^p\mathbf{H}^{i}(j_{\alpha,m}^*\mathcal{L})=0$ for $i\not=0.$ The functor $\operatorname{egf}$ is faithful by \cite[Proposition 5.3.9]{Pramod-2021}. Thus ${}^p\mathbf{H}^{i}(j_{\alpha,m}^*\mathcal{L}_0)=0$ for $i\not=0,$ hence $j_{\alpha,m}^*\mathcal{L}_0$ is perverse. If we denote the closed embedding $i:X(\alpha,m)-X(\alpha)\rightarrow X(\alpha,m),$ As well as $\operatorname{egf}({}^p\bfH^{-k} (i^!j_{\alpha,m}^*\mathcal{L}_0))=0,$ $\operatorname{egf}({}^p\bfH^{k}( i^*j_{\alpha,m}^*\mathcal{L}_0))=0,$ for $k\geq 0.$ Thus ${}^p\tau^{\leq 0} i^!j_{\alpha,m}^*\mathcal{L}_0=0,$ ${}^p\tau^{\geq 0} i^*j_{\alpha,m}^*\mathcal{L}_0=0.$ By \cite[Lemma 3.3.4]{Pramod-2021}\cite[Lemma 6.1]{laszlo2009}, it follows that $j_{\alpha,m}^*\mathcal{L}_0=\oplus_{\chi\in X(S_m)}j_{\alpha,m}^*IC(X(\alpha)_0,L_{\chi})^{\oplus s_{\chi}}.$ Thus it follows that $j_{\alpha,m}^*IC(X(\alpha)_0,L_{\chi})$ is pointwise pure. By the same argument, $j_{\alpha',m'}^*IC(X(\alpha')_0,L_{\chi'})$ are pointwise pure for any $\alpha',m'.$
 
 We now prove that for $X(\alpha')_0\subset \overline{X(\alpha)_0},$ $\ml_0$ is pointwise pure on $X(\alpha')_0.$ For $x\in X(\alpha')_0(\bbF_{q^n})\subset X(\alpha',m')_0(\bbF_{q^n}),$ where $(\alpha',m')<(\alpha,m),$ denote the corresponding representation by $M_x$. By Theorem~\ref{th}, and since the category of lisse sheaves (isomorphic to $\overline{\bbq_l}[\pi_1(X(\alpha'))]-\operatorname{mod}^{fc}$) is closed under kernels, cokernels, and extensions (by reducing to a finite extension $E$ of $\bbQ_l$ and its integral ring $o$, and considering $E\otimes_{o}\Ext_{\cD_c^b(X(\alpha'),o)}(L_1,L_2)$ for $o$-lisse sheaf $L_1,L_2.$ For $E\otimes_o L_{torsion}=0,$ we can choose $L_1,L_2$ torsion free), we have ${j_{\alpha'}}^*\mathcal{L}\in \cD^b_{\bfG_{\bfV},loc}(X(\alpha')).$ Using fpqc property as in Remark~\ref{rem}, $j_{\alpha'}^*\mathcal{L}_0$ is lisse sheaf complex whose simple constituents of $\mathbf{H}^i(j_{\alpha'}^*\mathcal{L}_0)$ have the form $L_{\chi'}\boxtimes \overline{\bbQ_l}(\zeta)$, where $\overline{\bbQ_l}(\zeta)$ is pure of weight $l$ on $\operatorname{Spec}\bbF_q$ for some $l\in \bbZ$ and $\operatorname{egf}(\overline{\bbQ_l}(\zeta))=\overline{\bbQ_l},$ $\chi\in X(S_{m'})$. By \cite[Theorem 3.4.1]{PD}, any lisse sheaf (local system) has a weight filtration, $0=L_0\subset L_1\subset\cdots\subset L_n=L$, where $L_i/L_{i-1}$ is a pure lisse sheaf and $wt(L_i/L_{i-1})<wt(L_{i+1}/L_i)$ for $i=1,\cdots,n-1.$ Moreover, any $\bbF_{q^l}$-point gives the same weight on $L_i/L_{i-1}.$ Thus, if for $\bbF_{q^n}$-point $x,$ $(L_{i}/L_{i-1})_x$ is of weight $0,$ then $(L_{i}/L_{i-1})$ is pointwise pure of weight $0.$ Indeed, $X(\alpha')$ is irreducible.

 By \cite[Proposition 4.14]{Lusztig1992}, $X(\alpha')_0$ is smooth with dimension $m'+\dim\mo_x.$ We consider the following map,
 $$\phi:\oplus_{j\in I}\Hom_k(\bfV_i,\bfV_i)/\Hom_{kQ}(M_x,M_x)\rightarrow \bfE_{\bfV}$$ defined by $g\rightarrow gx-xg.$

 Write $\bfV=\oplus \bfV_s^P\oplus \bfV_1\oplus \cdots \oplus \bfV_{m'}\oplus \oplus\bfV_r^I$ for the decomposition $M_x=P_{\alpha'}\oplus r_1\oplus \cdots \oplus r_{m'}\oplus I_{\alpha'},$ with $M_{x|_{\bfV_s^P}}=P_s^{{\alpha'_s}^P},$ $M_{x|_{\bfV_r^I}}=I_r^{{\alpha'_r}^I}$ and $M_x|_{\bfV_i}=r_i.$ After replacing $n$ if necessary, $x$ is decomposed as above over the field $\bbF_{q^n}.$ We choose a basis of $\bfV_i$ such that $x|_{\bfV_i}=(x_\alpha,x_\beta)=(1,t_i),$ where $\alpha,\beta$ are arrows, and consider the space $T_i:=k(0,1)\subset \bfE_{\bfV_i}.$ Denote $\bfE_{\bfV.\bfV'}=\oplus_{h}\Hom_k(\bfV_{s(h)},\bfV'_{t(h)}).$

 Let $\mathcal{T}=\im \phi\times \Pi_{i=1}^{m'}T_i\subset \bfE_{\bfV},$ with the embedding $(x,x_1,\cdots,x_{m'})\rightarrow x+\sum_{i=1}^{m'} x_i$ and consider $x+\mathcal{T}$. It is straightforward to see that $x+\mathcal{T}$ is the tangent space of $X(\alpha')$ at $x.$ And $\mathcal{T}$ contains $\Pi_{s\leq s'}\bfE_{\bfV^P_{s},\bfV^P_{s'}}\times \Pi_{s,i}\bfE_{\bfV_s^P,\bfV_i}\times\Pi_{j,r}\bfE_{\bfV_j,\bfV^I_r}\times \Pi_{s,r}\bfE_{\bfV^P_s,\bfV^I_r}\times\Pi_{r\geq r'}\bfE_{\bfV^I_{r},\bfV^I_{r'}}\times \Pi_{i\leq j}\bfE_{\bfV_i,\bfV_j}.$

 We only need to show that $\bfE_{\bfV_i,\bfV_i}\subset \mathcal{T},$ for other sets are contained in $\im \phi.$

 For $T_i$ and $\phi(\oplus_{j\in I}\Hom_k({\bfV_i}_j,{\bfV_i}_j)/\Hom_{kQ}({r_i},r_i))$ span $\bfE_{\bfV_i,\bfV_i},$ this completes the proof.

 Consider the map $\lambda(t)\in \bfG_{\bfV}$ as $\lambda(t)|_{\bfV_s^P}=t^{a_s},$ $\lambda(t)|_{\bfV_r^I}=t^{b_r},$ $\lambda(t)|_{\bfV_i}=t^{c_i},$ such that $a_s<c_i<b_r$ and $a_s<a_{s'}$ for $s<s'$, $b_r>b_{r'}$ for $r<r'$, $c_1<c_2<\cdots<c_{m'}.$ Then $\lambda(t)\im\phi\subset \im\phi$ and $\lambda(t)(T_i)\subset T_i.$ Thus $\mathcal{T}$ is $\lambda(t)$-stable.
 
 We then choose the $\lambda(t)$-stable complement of $\mathcal{T}$ in $\bfE_{\bfV}$ as $\mathcal{T}'.$ $$\mathcal{T}'\subset \Pi_{s> s'}\bfE_{\bfV^P_{s},\bfV^P_{s'}}\times \Pi_{s,i}\bfE_{\bfV_i,\bfV_s^P}\times\Pi_{j,r}\bfE_{\bfV^I_r,\bfV_j}\times \Pi_{s,r}\bfE_{\bfV^I_r,\bfV^P_s}\times\Pi_{r< r'}\bfE_{\bfV^I_{r},\bfV^I_{r'}}\times \Pi_{i> j}\bfE_{\bfV_i,\bfV_j}.$$ 
 
 We now consider $x+(\mathcal{T'}\times\Pi_{i=1}^{m'}T_i),$ it is transverse to the orbit $\mo_x$ at $x.$ 
 
 The action of $\lambda(t)$ preserves $\mathcal{T}$ and $\mathcal{T'}.$ And $\lambda(t)$ acts trivially on $x+T_i,$ $i=1,\cdots,m'.$
 
 Furthermore, $\lim_{t\rightarrow0}\lambda(t)u=0,$ for $u\in \mathcal{T'}.$ Thus fixed-point set $(x+(\mathcal{T'}\times\Pi_{i=1}^{m'}T_i))^{G_m}$ is $x+\Pi_{i=1}^{m'}T_i.$ 
 
 Following \cite[Lemma 10.5.2]{Pramod-2021}, when we consider $\bfG_{\bfV}\times (x+(\mathcal{T'}\times\Pi_{i=1}^{m'}T_i))\xrightarrow{f} \bfE_{\bfV},$ with $f(g,x')=g(x'),$ $x$ is a smooth point of $f.$ And the set of smooth points of $f,$ denoted as $\tilde{S},$ is open and stable under $\bfG_{\bfV}$ and $\operatorname{Fr}^n.$ Thus we can find an open subset $S$ of $(x+(\mathcal{T'}\times\Pi_{i=1}^{m'}T_i))$ such that $\tilde{S}\cong \bfG_{\bfV}\times S\xrightarrow{f} \bfE_{\bfV}$ is smooth with $S$ stable under $\operatorname{Fr}^n.$ Since $G_m$ acts on $(x+(\mathcal{T'}\times\Pi_{i=1}^{m'}T_i))$ and $G_m\subset \bfG_{\bfV},$ it follows that $S$ is $G_m$-stable. We now consider $U=X(\alpha')\cap S\cap (x+\Pi_{i=1}^{m'}T_i),$ one checks that $U$ is an open subset of $(x+\Pi_{i=1}^{m'}T_i),$ thus $U$ is smooth.
 
 Let $U'=S\cap (U\times\mathcal{T'}).$ It is an open subset of $(x+(\mathcal{T'}\times\Pi_{i=1}^{m'}T_i))$ and stable under $G_m,$ with fixed-point set $U.$ Then $\bfG_{\bfV}\times U'\xrightarrow{f} \bfE_{\bfV}$ is smooth morphism with relative dimension $\dim\bfG-\dim \mo_x.$ Let $U'\xrightarrow{i}\bfG_{\bfV}\times U'$ with $i(u)=(1,u).$ As in \cite[Proposition 5.7.3, Theorem 6.5.10]{Pramod-2021}, it follows that $i^*[-\dim \bfG_{\bfV}](\frac{-\dim \bfG_{\bfV}}{2})=i^![\dim \bfG_{\bfV}](\frac{\dim \bfG_{\bfV}}{2})$ is equivalence of categories. Furthermore, $i^*$ takes pure sheaves to pure sheaves. Thus $i^*f^*[-\dim \mo_x](\frac{-\dim \mo_x}{2})$ takes pure sheaves to pure sheaves. We now consider inclusion map $e:U\rightarrow U'$ and $p:U'\rightarrow U$ with $p(u)=\lim_{t\rightarrow0}\lambda(t)u$ as in \cite[Theorem 2.10.3]{Pramod-2021}.

 If we denote the inclusion map $i_{U'}:U'\rightarrow \bfE_{\bfV}$ and $i_{U'}=f\circ i,$ then $i_{U'}$ is $G_m$-equivariant. By \cite{BBD}, for proper map $p,$ $p_!$ preserves the weight. Thus $\ml_0$ is pure of weight $0$. Thus $i^*f^*(\mathcal{L}_0)$ is $G_m$-equivariant and pure of weight $0$. By \cite[Theorem 2.10.3]{Pramod-2021}, it follows that $e^*i^*f^*(\mathcal{L}_0)\cong p_*i^*f^*(\mathcal{L}_0).$ Because $e^*$ preserves weight $\leq 0$ and $p_*$ preserves weight $\geq 0,$ $e^*i^*f^*(\mathcal{L}_0)$ is pure. Since $U$ is smooth, considering $U\xrightarrow{i_{u}} X(\alpha)_0,$ it follows that $e^*i^*f^*(\mathcal{L}_0)\in \cD^b_{loc,m}(U).$ Thus by \cite[6.2.5]{PD}, it follows that $e^*i^*f^*(\mathcal{L}_0)$ is pointwise pure of weight $0$. Thus $j_{\alpha}^*\mathcal{L}_0$ is pointwise pure of weight $0$.

 Now consider the subset $X(\alpha',m')_0\subset \overline{X(\alpha)_0}.$ If we denote the open embedding $j^0:X(\alpha')_0\rightarrow X(\alpha',m')_0,$ since ${j^0}^*$ preserves the perverse $t$-structure and weight and ${}^p\mathbf{H}^n(\mathcal{F})=\mathbf{H}^{n-\dim(X(\alpha'))}(\mathcal{F})[\dim X(\alpha')]$ for $\mathcal{F}\in \cD^b_{\bfG_0,loc}(X(\alpha')_0)$, by Theorem~\ref{th} and the faithfulness of the functor $\operatorname{egf}$ on $\operatorname{Perv}(X_0,\overline{\bbQ_l}),$ it follows that ${}^p\mathbf{H}^l(j_{\alpha',m'}^*\mathcal{L}_0)$ is pure of weight $l.$ In the triangle $${}^p\tau^{<l}(j_{\alpha',m'}^*\mathcal{L}_0)\rightarrow {}^p\tau^{\leq l}(j_{\alpha',m'}^*\mathcal{L}_0)\rightarrow {}^p\mathbf{H}^l(j_{\alpha',m'}^*\mathcal{L}_0)[-l]\xrightarrow{+1},$$ if ${}^p\tau^{<l}(j_{\alpha',m'}^*\mathcal{L}_0)$ and ${}^p\mathbf{H}^l(j_{\alpha',m'}^*\mathcal{L}_0)[-l]$ are pure of weight $0,$ then ${}^p\tau^{\leq l}(j_{\alpha',m'}^*\mathcal{L}_0)$ is pure of weight $0.$ Thus by induction on $l,$ $j_{\alpha',m'}^*\mathcal{L}_0$ is pure of weight zero and $j_{\alpha',m'}^*\mathcal{L}=\oplus_l {}^p\mathbf{H}^l(j_{\alpha',m'}^*\mathcal{L})[-l]$ has the form $\oplus_{l,\chi'}j_{\alpha',m'}^*(IC(X(\alpha'),L_{\chi'}))^{\oplus s_{l,\chi'}}[l].$ 

 Let closed embedding $i':X(\alpha',m')_0-X(\alpha')_0\rightarrow X(\alpha',m')_0$ and consider ${}^p\tau^{\leq 0}{i'}^!{}^p\bfH^l(j^*_{\alpha',m'}\mathcal{L}_0)$ and ${}^p\tau^{\geq 0}{i'}^*{}^p\bfH^l(j^*_{\alpha',m'}\mathcal{L}_0).$ After applying $\operatorname{egf},$ ${}^p\bfH^{-k}({i'}^!{}^p\bfH^l(j^*_{\alpha',m'}\mathcal{L}_0))$ and ${}^p\bfH^{k}({i'}^*{}^p\bfH^l(j^*_{\alpha',m'}\mathcal{L}_0))$ are $0$ for $k\geq 0.$ By the faithfulness of $\operatorname{egf},$ it follows that ${i'}^*{}^p\bfH^l(j^*_{\alpha',m'}\mathcal{L}_0)\in {}^p\cD^{\leq -1}$ and ${i'}^!{}^p\bfH^l(j^*_{\alpha',m'}\mathcal{L}_0)\in {}^p\cD^{\geq 1},$ by \cite[Lemma 3.3.4]{Pramod-2021} and \cite[Lemma 6.1]{laszlo2009}, ${}^p\bfH^l(j^*_{\alpha',m'}\mathcal{L}_0)=j^0_{!*}({}^p\bfH^l({j_{\alpha'}}^*\mathcal{L}_0)).$
 By \cite[Lemma 3.4.3]{Pramod-2021}, and the faithfulness of $\operatorname{egf},$ any simple perverse constituent of ${}^p\mathbf{H}^l(j_{\alpha',m'}^*\mathcal{L}_0)$ has the form $j_{\alpha',m'}^*(IC(X(\alpha')_0,L_{\chi}))\boxtimes \overline{\bbQ_l}(\zeta),$ where $\overline{\bbQ_l}(\zeta)$ is pure of weight $l$ on $\operatorname{Spec}\bbF_q$ and $\operatorname{egf}(\overline{\bbQ_l}(\zeta))=\overline{\bbQ_l}.$ Moreover, $j_{\alpha',m'}^*(IC(X(\alpha')_0,L_{\chi})$ is pointwise pure of weight $0$ and $\boxtimes \overline{\bbQ_l}(\zeta)$ preserves pointwise purity and purity. If $A\rightarrow B\rightarrow C\xrightarrow{+1}$ is a triangle and $A$ and $C$ are pointwise pure of weight $a$, then $B$ is pointwise pure of weight $a$. Thus $j_{\alpha',m'}^*\mathcal{L}_0$ is pointwise pure. 

This completes the proof that $\mathcal{L}_0$ is pointwise pure. 
\end{proof} 
\subsection*{Acknowledgements} 
Y. Wu and J. Xiao are partially supported by National Natural Science Foundation of China [Grant No. 12471030].
\end{spacing}
\bibliography{ref}

\begin{thebibliography}{10}

\bibitem{Pramod-2021}
P.~N. Achar.
\newblock {\em Perverse sheaves and applications to representation theory}, volume 258 of {\em Mathematical Surveys and Monographs}.
\newblock American Mathematical Society, Providence, RI, 2021.

\bibitem{2002Crystal}
J.~Beck and H.~Nakajima.
\newblock Crystal bases and two-sided cells of quantum affine algebras.
\newblock {\em Duke Mathematical Journal}, 123(2):335--402, 2002.

\bibitem{BBD}
A.~A. Be\u{\i}linson, J.~Bernstein, and P.~Deligne.
\newblock Faisceaux pervers.
\newblock In {\em Analysis and topology on singular spaces, {I} ({L}uminy, 1981)}, volume 100 of {\em Ast\'{e}risque}, pages 5--171. Soc. Math. France, Paris, 1982.

\bibitem{10.1215/00127094-2681278}
J.~Brundan, A.~Kleshchev, and P.~J. McNamara.
\newblock {Homological properties of finite-type Khovanov–Lauda–Rouquier algebras}.
\newblock {\em Duke Mathematical Journal}, 163(7):1353 -- 1404, 2014.

\bibitem{articleff}
C.~Concini, G.~Lusztig, and C.~Procesi.
\newblock {Homology of the Zero-Set of a Nilpotent Vector Field on a Flag Manifold}.
\newblock {\em J. Amer. Math. Soc.}, 1:15--34, 01 1988.

\bibitem{PD}
P.~Deligne.
\newblock {La conjecture de Weil : II}.
\newblock {\em Mathématiques de l'IHÉS}, 52:137--252, 1980.

\bibitem{Jouanolou}
J.~Jouanolou.
\newblock Cohomologie de quelques schemas classiques et theorie cohomologique des classes de chern.
\newblock {\em In: Illusie, L. (eds) Séminaire de Géométrie Algébrique du Bois-Marie 1965–66 SGA 5. Lecture Notes in Mathematics.}, 589.

\bibitem{2012PBW}
S.~Kato.
\newblock {Poincaré–Birkhoff–Witt bases and Khovanov–Lauda–Rouquier algebras}.
\newblock {\em Duke Mathematical Journal}, 163(3):619 -- 663, 2014.

\bibitem{cd088214-009e-3813-9a86-97f9a299a24c}
S.~Kato.
\newblock An algebraic study of extension algebras.
\newblock {\em American Journal of Mathematics}, 139(3):567--615, 2017.

\bibitem{lan2025structurecoefficientsquantumgroups}
Y.~Lan, Y.~Wu, and J.~Xiao.
\newblock Structure coefficients for quantum groups, 2025.
\newblock arXiv 2510.25575.

\bibitem{laszlo2009}
Y.~Laszlo and M.~Olsson.
\newblock {Perverse t-structure on Artin stacks}.
\newblock {\em Mathematische Zeitschrift}, 261(4):737--748, 2009.

\bibitem{li2007notesaffinecanonicalmonomial}
Y.~Li.
\newblock Notes on affine canonical and monomial bases, 2007.
\newblock arXiv 0610449.

\bibitem{Li2006ARquiverAT}
Y.~Li and Z.~Lin.
\newblock {AR-quiver approach to affine canonical basis elements}.
\newblock {\em Journal of Algebra}, 318:562--588, 2006.

\bibitem{linxiao}
Z.~Lin, J.~Xiao, and G.~Zhang.
\newblock {Representations of Tame Quivers and Affine Canonical Bases}.
\newblock {\em Publ. Res. Inst. Math. Sci.}, 47(4):825–885, 2011.

\bibitem{lusztig1990canonical}
G.~Lusztig.
\newblock Canonical bases arising from quantized enveloping algebras.
\newblock {\em Journal of the American Mathematical Society}, 3(2):447--498, 1990.

\bibitem{Lusztig1992}
G.~Lusztig.
\newblock {Affine quivers and canonical bases}.
\newblock {\em Publications Math{\'e}matiques de l'Institut des Hautes {\'E}tudes Scientifiques}, 76(1):111--163, Dec 1992.

\bibitem{lusztig2010introduction}
G.~Lusztig.
\newblock {\em Introduction to quantum groups}, volume 110 of {\em Progress in Mathematics}.
\newblock Birkh\"{a}user Boston, Inc., Boston, MA, 1993.

\bibitem{lusztig1998canonical}
G.~Lusztig.
\newblock {Canonical bases and Hall algebras}.
\newblock In {\em Representation theories and algebraic geometry}, pages 365--399. Springer, 1998.

\bibitem{MCGERTY2005411}
K.~McGerty.
\newblock {The Kronecker quiver and bases of quantum affine sl2}.
\newblock {\em Advances in Mathematics}, 197(2):411--429, 2005.

\bibitem{McNamara2017}
P.~J. McNamara.
\newblock {Representations of Khovanov--Lauda--Rouquier algebras III: symmetric affine type}.
\newblock {\em Mathematische Zeitschrift}, 287(1):243--286, Oct 2017.

\bibitem{1980Etale}
J.~S. Milne.
\newblock {\em {Etale Cohomology (PMS-33)}}.
\newblock Princeton University Press, 1980.

\bibitem{1994Geometric}
D.~Mumford, J.~Fogarty, and F.~Kirwan.
\newblock Geometric invariant theory.
\newblock {\em Springer Verlag Berlin New York}, 1994.

\bibitem{wreo6318}
J.~A. Sauter.
\newblock Springer theory and the geometry of quiver flag varieties.
\newblock University of Leeds, September 2013.

\bibitem{shoji2026monomialbasescanonicalbases}
T.~Shoji and Z.~Zhou.
\newblock Monomial bases and canonical bases for quantum affine algebras, 2026.
\newblock arXiv 2605.14247.

\bibitem{Steinberg1976}
R.~Steinberg.
\newblock On the desingularization of the unipotent variety.
\newblock {\em Inventiones Mathematicae}, 36(1):209--224, Dec 1976.

\bibitem{sunarticle}
S.~Sun.
\newblock {Decomposition Theorem for Perverse sheaves on Artin stacks}.
\newblock {\em Duke Mathematical Journal}, 161, 09 2010.

\bibitem{vistoli2007notesgrothendiecktopologiesfibered}
A.~Vistoli.
\newblock {Notes on Grothendieck topologies, fibered categories and descent theory}, 2007.
\newblock arXiv math/0412512.

\bibitem{XIAO2023510}
J.~Xiao, H.~Xu, and M.~Zhao.
\newblock {Tame quivers and affine bases I: A Hall algebra approach to the canonical bases}.
\newblock {\em Journal of Algebra}, 633:510--562, 2023.

\bibitem{+2000+97+116}
P.~Zhang.
\newblock {PBW-basis for the composition algebra of the Kronecker algebra}.
\newblock {\em Journal für die reine und angewandte Mathematik}, 2000(527):97--116, 2000.

\end{thebibliography}
\end{document}